\documentclass{amsart} \usepackage{amscd} \usepackage{amssymb} \usepackage[all, knot]{xy} \usepackage{bbm}
\usepackage[top=1.2in, bottom=1.2in, left=1.2in, right=1.2in]{geometry}
\xyoption{all}
\xyoption{arc}
\usepackage{hyperref}
\usepackage[mathscr]{euscript}
\usepackage{extpfeil}

\usepackage{tikz-cd}

\usepackage{xcolor}

\makeatletter
\let\@wraptoccontribs\wraptoccontribs
\makeatother

\def\extpower{\bigwedge}

\def\tpsi{\tilde{\psi}}
\def\tQ{\tilde{Q}}
\def\tf{\tilde{f}}
\def\Pre{\operatorname{Pre}}
\def\HochII{\operatorname{Hoch}^{\mathrm{II}}}
\def\Hoch{\operatorname{Hoch}}

\def\Deltaleft{\Delta^\lft}
\def\Deltaright{\Delta^\rght}

\def\cPerf{\operatorname{{\mathcal P}erf}}

\def\SemiFreedg{\operatorname{SF}_\dg}
\def\HProjdg{\operatorname{HP}_\dg}
\def\projdg{\operatorname{proj}_\dg}

\def\grMod{\operatorname{grMod}}
\def\Nonreg{\operatorname{Nonreg}}
\def\mf{\operatorname{mf}}
\def\dg{{\mathrm{dg}}}
\def\mfdg{\operatorname{mf}_\dg}
\def\LFdg{\operatorname{LF}_\dg}

\def\Yon{\Upsilon}

\def\Kos{\operatorname{Kos}}

\def\Thick{\operatorname{Thick}}

\def\Lotimes{\otimes^{\bL}}

\def\LotimesII{\otimes^{\bL, II}}

\def\del{\partial}

\def\cube#1#2#3#4#5#6#7#8{
& #5 \ar[rr] \ar[dl] \ar@{-}[d] && #6 \ar[dd] \ar[dl] \\
#1 \ar[rr] \ar[dd]  & \ar[d] & #2 \ar[dd] \\
& #7 \ar@{-}[r] \ar[dl] & \ar[r] & #8 \ar[dl] \\
#3 \ar[rr] && #4 \\
}

\def\rght{{\mathrm{right}}}

\def\lft{{\mathrm{left}}}

\def\Tot{\operatorname{Tot}}

\def\chr{\operatorname{char}}

\def\sing{{\operatorname{sing}}}

\def\image{\operatorname{im}}
\def\im{\image}
\def\ker{\operatorname{ker}}
\def\cone{\operatorname{cone}}

\def\cA{\mathcal A}

\def\cC{\mathcal C}

\def\cD{\mathcal D}

\def\cO{\mathcal O}

\def\cS{\mathcal S}
\def\cT{\mathcal T}
\def\cM{\mathcal M}
\def\cN{\mathcal N}

\def\cV{\mathcal V}

\def\Sing{\operatorname{Sing}}
\def\Proj{\operatorname{Proj}}

\def\Tor{\operatorname{Tor}}

\def\End{\operatorname{End}}

\def\Spec{\operatorname{Spec}}

\def\into{\hookrightarrow}
\def\onto{\twoheadrightarrow}
\def\ann{\operatorname{ann}}
\def\cH{\mathcal{H}}

\def\a{\alpha}
\def\b{\beta}
\def\g{\gamma}

\def\ug{{\underline{g}}}

\def\CechC{{\check{\mathrm{C}}}}

\DeclareMathOperator*{\colim}{colim}

\newcommand{\Q}{\mathbb{Q}}

\newcommand{\bP}{\mathbb{P}}

\newcommand{\A}{\mathbb{A}}

\newcommand{\G}{\mathbb{G}}

\newcommand{\bL}{{\mathbb{L}}}
\newcommand{\R}{{\mathbb{R}}}

\newcommand{\Z}{\mathbb{Z}}

\newcommand{\fp}{{\mathfrak p}}
\newcommand{\fm}{{\mathfrak m}}

\newcommand{\fq}{{\mathfrak q}}

\numberwithin{equation}{section}

\theoremstyle{plain} 
\newtheorem{thm}[equation]{Theorem}
\newtheorem{thm-conj}[equation]{Theorem-Conjecture}
\newtheorem{defn-conj}[equation]{Definition-Conjecture}

\newtheorem*{introthm*}{Theorem}

\newtheorem{cor}[equation]{Corollary}
\newtheorem{lem}[equation]{Lemma}

\newtheorem{prop}[equation]{Proposition}

\theoremstyle{definition}
\newtheorem{defn}[equation]{Definition}

\newtheorem{ex}[equation]{Example}

\theoremstyle{remark}
\newtheorem{rem}[equation]{Remark}

\def\Perfdg{\operatorname{Perf}_{{\mathrm{dg}}}}

\def\qPerfcdg{\operatorname{qPerf}_{{\mathrm{cdg}}}}

\def\SemiProjdg{\operatorname{SP}_{{\mathrm{dg}}}}

\newcommand{\Cech}{\v{C}ech }

\newcommand{\Hom}{\operatorname{Hom}}

\newcommand{\supp}{\operatorname{supp}}

\newcommand{\xra}[1]{\xrightarrow{#1}}

\newcommand{\xira}[1]{\xhookrightarrow{#1}}
\newcommand{\xla}[1]{\xleftarrow{#1}}

\newcommand{\id}{\operatorname{id}}
\newcommand{\pd}{\operatorname{pd}}

\def\l{\lambda}
\def\L{{\mathbb L}}

\def\and{ \text{ and } }
\def\wideand{ \, \, \text{ and }  \, \, }

\def\can{\mathrm{can}}

\def\ob{\operatorname{ob}}

\def\op{{\mathrm{op}}}

\def\G{\Gamma}

\def\nat{\natural}

\def\Mod{\mathrm{Mod}}
\def\Moddg{\Mod_{{\mathrm{dg}}}}

\begin{document}

\title{On the Hochschild homology of curved algebras}

\author[Walker]{Mark E. Walker}
\address{Department of Mathematics, University of Nebraska, Lincoln, NE 68588, U.S.A.}
\email{mark.walker@unl.edu}

\contrib[with an appendix by]{Benjamin Briggs}

\begin{abstract}
  We compute the Hochschild homology of the differential graded category of perfect curved modules over suitable curved rings, giving what might be termed ``de Rham models'' for such.  This represents a
  generalization of previous results by Dyckerhoff, Efimov, Polishchuk, and Positselski
  concerning the Hochschild homology of matrix factorizations.
  A key ingredient in the proof is a theorem due to B.~Briggs, which represents a ``curved version'' of a celebrated
  theorem of Hopkins and Neeman. The proof of Briggs' Theorem is included in an appendix to this paper.
\end{abstract}

\date{\today}


\maketitle

\setcounter{tocdepth}{1}
\tableofcontents

\thanks{
  Walker was partly supported by the National Science Foundation (NSF) grant DMS-2200732 and Briggs was funded by the European Union under the Grant Agreement no.\ 101064551 (Hochschild).  Part of this paper was written while Briggs and Walker 
were in residence at the Simons Laufer Mathematical Sciences Institute  in Berkeley, California, during the Spring 2024 semester,
and they were supported by the  NSF under Grant No. DMS-1928930 and by the Alfred P. Sloan Foundation under grant G-2021-16778.
}

\section{Introduction}\label{introduction}

In this paper, by a {\em curved ring} we mean a pair $\cA = (A, w)$, where $A$ is a commutative $\Z$-graded ring that is concentrated in even degrees (i.e., $A^i = 0$ for $i$ odd)
and $w$, known as the {\em
  curvature}, is an element of cohomological degree $2$: $w \in A^2$. 
A {\em curved module} over $\cA$ is a pair $(M, \del)$ consisting of a  graded $A$-module $M$ equipped with an $A$-linear endomorphism $\del$
of degree $1$ satisfying $\del^2(m) = w \cdot m$ for all $m \in M$. 
The collection of curved $\cA$-modules form a differential graded (dg) category, written as $\Moddg(\cA)$. 
A curved module $(M, \del)$ will be called  {\em perfect} if $M$ is  finitely generated and projective as a graded $A$-module.  
We write $\Perfdg(\cA)$ for the full dg subcategory on the collection of all perfect curved $\cA$-modules. 

The main Theorem of this paper computes the Hochschild homology of $\Perfdg(\cA)$ over a suitable ground ring:

\begin{thm} \label{MainTheorem}
  Let $k$ be a $\Z$-graded commutative ring that is concentrated in even degrees, regular, and excellent, 
  and let $\cA = (A, w)$ be a curved ring such that $A$ is an essentially smooth $k$-algebra.
  Then there is an isomorphism
  $$
  HH(\Perfdg(\cA)) \cong \R\G_{\Nonreg(\cA)} (\Omega^\cdot_{A/k}, dw)
  $$
  in the derived category of dg $A$-modules.
  
  Moreover,   when $k$ is an algebra over the  field of rational numbers, 
  this isomorphism is represented by a natural quasi-isomorphism of explicit dg $k$-modules, under which the Connes' operator
on  Hochschild homology corresponds with the de Rham differential on $\Omega^\cdot_{A/k}$.
\end{thm}

Let us explain some of the terminology in this theorem.
$HH(\Perfdg(\cA))$ denotes the Hochschild homology, taken relative to
$k$,  of the  $k$-linear dg category $\Perfdg(\cA)$. It is defined in a manner analogous to the formula
$HH(R) = R \Lotimes_{R \otimes_k R} R$ for an ordinary $k$-algebra $R$;  see \eqref{E421} below for the precise definition. We will sometimes write it as $HH(\Perfdg(\cA)/k)$ when the role of
the base ring needs emphasis. 
By $(\Omega^\cdot_{A/k}, dw)$ we mean the dg $A$-module consiting of the graded $A$-module $\bigoplus_q \Sigma^q \Omega^q_{A/k}$ equipped with the differential
given as multiplication by $dw$. 
(Here, $\Omega^1_{A/k}$ is the module of Kahler differentials, equipped with the grading induced by that on $A$,
$\Omega^q_{A/k}$ is its $q$-th exterior power over $A$ and $\Sigma^q$ refers to shifting the grading by $q$.)
Equivalently, $(\Omega^\cdot_{A/k}, dw)$ is the totalization of the complex 
$$
\cdots \to 0 \to A \xra{dw} \Sigma^2 \Omega^1_{A/k}\xra{dw} \Sigma^4\Omega^2_{A/k} \xra{dw} \cdots \xra{dw} \Sigma^{2d} \Omega^{d}_{A/k} \to 0 \to \cdots
$$
of graded $A$-modules.
By $\Nonreg(\cA)$ we mean the Zariski closed subset $\{\fp \in \Spec(A) \mid w \in \fp^2 A_\fp\}$ of $\Spec(A)$ and $\R\G_{\Nonreg(\cA)}$ refers
to taking local cohomology along it.

As a simple example, if the curvature is trivial, $w = 0$, then $\Perfdg(\cA)$ is the dg category of perfect dg $A$-modules, $\Nonreg(\cA) = \Spec(A)$, and thus our results
yield the (previously well-known) isomorphism
  $$
  HH(\Perfdg(A)) \cong \bigoplus_j \Sigma^j \Omega^j_{A/k}.
  $$
  Of greater interest is when $w$ is a non-zero-divisor in $A$ (i.e., it is nonzero on each component of the regular scheme $\Spec(A)$).
    In this case,   $\Nonreg(\cA)$ coincides with
  $$
  \Nonreg(A/w) = \{\fp \in \Spec(A) \mid w \in \fp \text{ and $(A/w)_{\fp}$ is not a regular ring} \},
  $$
  the non-regular locus of the hypersurface $A/w$.
  For the rest of the introduction, we assume $w$ is a non-zero-divisor.

\subsection{The case when $k$ is a perfect field}
When $k$ is a perfect field (and $w$ is a non-zero-divisor) we have
  $$
  \Nonreg(A/w) = \Sing_k(A/w) := \{\fp \in \Spec(A/w) \mid \text{$A/w$ is not smooth over $k$ near $\fp$}\}.
  $$
Let us write $\Sing(w)$ for the singular locus of the morphism
  $\Spec(A) \xra{w} \A^1_k = \Spec(k[x])$ of affine $k$-schemes induced by the $k$-algebra map $k[x] \to A$ sending $x$ to $w$.
  The Jacobian criterion gives that $\Sing(w)$ coincides with the support of $(\Omega^\cdot_{A/k}, dw)$, and so we
  have
  $$
  \Nonreg(A/w) = \Sing(w) \cap V(w) = \supp (\Omega^\cdot_{A/k}, dw) \cap V(w).
  $$
We may thus deduce the following consequence of Theorem  \ref{MainTheorem}.

  \begin{cor} \label{CorPerfectField} Under the assumptions of Theorem
    \ref{MainTheorem}, suppose in addition that $w$ is a
    non-zero-divisor and $k$ is a perfect field.
    Then there are  isomorphism in the derived category
    $$
HH(\Perfdg(\cA)) \cong  \R\Gamma_{V(w)} (\Omega^\cdot_{A/k}, dw).
$$
of dg $A$-modules. 
In particular, if $\Sing(w) \subseteq V(w)$, then
$$
HH(\Perfdg(\cA)) \cong  (\Omega^\cdot_{A/k}, dw).
$$
\end{cor}

\subsection{Main Examples}
We now discuss an important collection of examples. Throughout the rest of this introduction, we fix a field
$F$, an essentially smooth (ungraded) $F$-algebra $Q$, 
and elements $f_1, \dots, f_c$ in $Q$. Set $\tQ := Q[t_1, \dots, t_c]$ for degree two indeterminants $t_1, \dots, t_c$ and $\tf := \sum_i f_i t_i \in \tQ$. 
Then $(\tQ, \tf)$ is a curved algebra and it is essentially smooth over both $F$ and $F[t_1, \dots, t_c]$.
In this situation we have  $\Sing(\A^c_Q \xra{\tf} \A^1_F) \subseteq V(\tf)$, because $\tf = \sum_i \frac{\partial \tf}{\partial t_i} t_i$,. Corollary \ref{CorPerfectField} thus
specializes to give:

\begin{cor} \label{cor76}  With the notation just introduced, assume also that $F$ is a perfect field.
  Then there is an isomorphism
  $$
  HH(\Perfdg(\tQ, \tf)/F) \xra{\cong} (\Omega^\cdot_{\tQ/F}, d \tf).
  $$
  in the derived category of $\tQ$-modules. 
  Moreover, if $\chr(F) = 0$, this isomorphism may be realized as a natural quasi-isomorphism of dg $F$-algebras, under which the Connes'  operator corresponds
  with the de Rham differential.
\end{cor}

Corollary \ref{cor76} allows for a calculation of the Hochschild homology of the bounded derived category of modules over a complete intersection.
In detail, in the setting of this Corollary, suppose $f_1, \dots, f_c$ form a regular sequence and set $R = Q/(f_1, \dots, f_c)$.
Then it is known that the dg category $\Perfdg(\tQ, \tf)$ is quasi-equivalent to $\cD^b_\dg(R)$, a dg enhancement of the
bounded derived category of $R$-modules; see \cite{Martin}.  Since Hochschild homology in invariant under quasi-equivalence, Corollary \ref{cor76} gives
an isomorphism
$$
HH(\cD^b_\dg(R)/F) \cong (\Omega^\cdot_{\tQ/F}, d\tf)
$$
in the derived category of $Q$-modules.

\subsection{An example where $k$ is not a field}
In our running example, it also makes sense to take $k = F[t_1 \dots, t_c]$.
Then $\Omega^1_{\tQ/k} \cong \Omega^1_{Q/F}[t_1, \dots, t_c]$ and $d \tf = \sum_i t_i df_i$, which leads to a simpler complex $(\Omega^\cdot_{\tQ/k}, dw)$.
But the singular  locus of $\A^c_Q \xra{\tf} \A^1_k= \A^c_F \times_F \A^1$ often fails to be contained in $V(\tf)$,
making the role of local cohomology meaningful. 
For instance if $Q = F[x_1, \dots, x_n]$, or any localization of such, the singular locus of the morphism $\A^{c+n}_F \cong \Spec(\tQ)\xra{\tf} \A^1_k \cong \A^{c+1}_F$
is the set of solutions of $(t_1, \dots, t_c)  \cdot J = (0, \dots, 0)$
where $J$ is the $c \times n$ Jacobian matrix $(\partial f_i/\partial x_j)$; this is typically not contained in $V(\tf)$.

\subsection{The codimension one case and  matrix factorizations}
Let us specialize our running example  to the case $c = 1$, so that $\cA = (Q[t], ft)$ and $R = Q/f$ is a hypersurface.  Taking $k = F[t]$,  we have 
$$
HH((Q[t], ft)/F[t]) \cong \R\G_{\Nonreg(Q[t]/ft)} (\Omega^\cdot_{Q/k}[t], t df).
$$
An important  variant of this is given by the curved algebra  $(Q[t, t^{-1}], ft)$ over the base ring $k = F[t, t^{-1}]$.
In this case, we may identify $\Perfdg(Q[t, t^{-1}], ft)$ with the $\Z/2$-graded
dg category of {\em matrix factorizations} of $f \in Q$, written $\mf_\dg(Q,f)$. The objects are $\Z/2$-graded
finitely generated projective $Q$-modules equipped with odd degree endomorphisms that square to $f$, and the hom sets are $\Z/2$-graded complexes of $Q$-modules.
The associated homotopy category is isomorphic to the {\em singularity
  category} of $R = Q/f$; see  \cite[Theorem 4.4.1]{MCMbook} and \cite[Theorem 3.9]{Orlov}.
In general, a $\Z/2$-graded $F$-linear dg category $\cC$ is the same thing as a $F[t, t^{-1}]$ linear dg category where $t$ is an indeterminant of degree $2$.
We will write $HH_{\Z/2}(\cC)$ for $HH(\cC/F[t, t^{-1}])$.
Our results thus specialize to give:

\begin{cor} \label{IntroCor10}
  Let $F$ be a  field, $Q$ an essentially smooth $F$-algebra, and $f \in Q$ a non-zero-divisor.
  \begin{enumerate}
    \item If $F$ is perfect and the only singular value of the morphism $\Spec(Q) \xra{f} \A^1_{F}$ is the origin,
      then we have an isomorphism
    $$
    HH_{\Z/2}(\mfdg(Q,f)) \cong(\Omega^{\Z/2}_{Q/F}, df)
    $$
    where the superscript $\Z/2$ indicates  forming the $\Z/2$-folding of this complex.
    
  \item More generally, there exists an isomorphism
$$
HH_{\Z/2}(\mfdg(Q,f)) \cong
\R \Gamma_{V(f)}  (\Omega^{\Z/2}_{Q/F}, df)
= \Sigma^{-1} \cone \left((\Omega^{\Z/2}_{Q/F}, df) \into (\Omega^{\Z/2}_{Q/F}[1/f], df)\right).
$$
\end{enumerate}
\end{cor}

\begin{ex} \label{ex76} Suppose $F$ is a perfect field and $f \in F[x_1, \dots, x_n]$ is such that the only singular value of the
  morphism $f: \A^n_F  \to \A^1_F$  is at the origin. 
  Then $\frac{\partial f}{x_1}, \dots, \frac{\partial f}{x_n}$ form a regular sequence and hence 
  the Corollary gives an isomorphism
  $$
  HH_{\Z/2}(\mfdg(F[x_1, \dots, x_n],f)) \cong \Sigma^{-n}
  \left(\frac{F[x_1, \dots, x_n]}{(\frac{\partial f}{x_1}, \dots, \frac{\partial f}{x_n})}\right)
  $$
  in the derived category. 
  \end{ex}

If the morphism $\Spec(Q) \xra{f}\A^1_F$ has only a finite number of singular values, which is always the case when $\chr(F) = 0$,
then, up to quasi-isomorphism, the dg $A$-module $(\Omega^\cdot_{Q/F}, df)$  is a finite direct sum indexed by these values, and
$\R\G_{V(f)} (\Omega^\cdot_{Q/F}, df)$ is the summand corresponding to the origin. In other words, in this situation we may localize $Q$ a bit, without
affecting the dg category $\mfdg(Q,f)$ up to quasi-equivalence, and we will arrive at a situation in which the local cohomology functor $\R \Gamma_{V(f)}$ is unneeded.
 
In characteristic $p > 0$ things can be more complicated, and indeed $HH(\mfdg(Q,f))$ need not even have finitely generated homology.
For instance if $Q = F[x]$ and $f = x^p$ with $p = \chr(F)$, then every point of $\A^1_F$ is a singular value of the morphism $\Spec(Q) = \A^1_F \xra{z \mapsto z^p} \A^1_F$.
In this case the Corollary gives 
$$
HH_{\Z/2}(\mfdg(Q,f)) \cong \G_{(x)} (F[x] \xra{0} F[x] dx) =  \left(\frac{F[x, x^{-1}]}{F[x]} \xra{0} \frac{F[x, x^{-1}]}{F[x]} dx\right),
$$
and the homology of $HH(\Perfdg(\cA))$ is not finitely generated as a $Q$-module.  This is related to the fact that $\mfdg(Q,f)$ is not a ``homologically smooth''
$\Z/2$-graded  category over $F$. Indeed, in Theorem \ref{smooththm} below, we show:

\begin{thm} \label{introsmooththm} Let $Q$ be an essentially smooth $F$-algebra with $F$ a perfect field
  and assume $f$ is a non-zero-divisor in $Q$. 
The $\Z/2$-graded dg category $\mfdg(Q,f)$ is homologically smooth over $F$
  if and only if the origin in $\A^1_F$ is an isolated singular value of the morphism of affine schemes $\Spec(Q) \xra{f} \A^1_F$.
\end{thm}

Finally, we give an  example showing how badly things can behave for imperfect fields:
Suppose $\chr(F) = p > 0$ and
there exists an element $a \in F \setminus F^p$. Take $Q = F[x, y]$
and $f = y(x^p - a)$ and set $\fm = (y, x^p -a)$, a maximal ideal of $Q$ whose residue field is the purely inseparable field extension
$F(\sqrt[p]{a})$ of $F$. 
Then  $\Nonreg(Q/f) = \{\fm\}$ and $\Sing(f) = V(x^p - a) \subseteq f^{-1}(0)$.
In particular, the only singular value of the  morphism $\Spec(Q) \xra{f} \A^1_F$ is the origin.
But, as $df = dy(x^p-a)$, the complex $(\Omega^{\Z/2}_Q, df)$ may be identified with the $\Z/2$-folding of the complex
$$
(F[x,y] \xra{x^p-a} F[x,y]) \oplus \Sigma^{-1} (F[x,y] \xra{x^p-a} F[x,y]). 
$$
This complex is not supported on $V(f)$ and thus
$HH_{\Z/2}(mf(Q,f))$ and $(\Omega^{\Z/2}_{Q/F}, df)$ differ.

\subsection{Connections with other results}
Some of the results presented here were known before, at least in some special cases:
Polishchuk and Positselski established part (1) of Corollary \ref{IntroCor10} in \cite[Sec 4.8]{PP}; in fact, they establish
such a result so long as there is an $F$-smooth subvariety $Z$ of $f^{-1}(0)$ such that
$f$ is smooth at each point of $\Spec(Q) \setminus Z$. Their result was generalized by Efimov and Positselski by allowing $Z$ to admit an $F$-smooth stratification \cite[Section
B.1.6]{EP}.
The special case of  Corollary \ref{IntroCor10} in which $f$ has just one singular point was established by Dyckerhoff \cite[Theorem 6.6]{Dyck}; i.e.,  he essentially established
Example \ref{ex76}.  See \cite{PV} for an equivariant version of this result.
Papers \cite{Ef}, \cite{CT} \cite{LP}, \cite{Segal} and \cite{Shk} contain related results.

\subsection*{Acknowledgements} The author thanks Michael Brown for many extensive conversations on the topic of this paper and the details of its development.

\section{Outline of the proofs of the main results} \label{sec:outline}

In this section we outline the proofs of the results announced in the introduction; the details are contained in the subsequent sections. We refer the reader to Section
\ref{sec:homalg} for any undefined terms appearing in this section.

For a $k$-linear dg category $C$ we write $C^\op$ for its {\em opposite} dg category and $C^e = C \otimes_k C^\op$ for its {\em enveloping} dg category. Objects of the latter are
ordered pairs $(X, Y)$ of objects of $C$ with hom complexes given by $\Hom_{C^e}((X,Y), (X',Y')) = \Hom_C(X, X') \otimes_k  \Hom_C(Y', Y)$. 
We write $\Deltaright_C$ (resp. $\Deltaleft_C$) for the right (resp. left) dg $C^e$-modules given on objects by 
$\Deltaright_C(X,Y) = \Hom_C(Y,X)$ (resp. $\Deltaleft_C(X,Y) = \Hom_C(X,Y)$). The Hochschild homology complex (over $k$) of $C$ is by definition the derived tensor product of these two modules:
\begin{equation} \label{E421}
HH(C) = HH(C/k) := \Deltaright_C \Lotimes_{C^e} \Deltaleft_C
\end{equation}
which is an object in the derived category of dg $k$-modules.

For a curved ring $\cA = (A, w)$, its {\em opposite} is the curved ring $\cA^\op := (A, -w)$, and 
its {\em enveloping algebra} (over $k$)  is $\cA^e := \cA \otimes_k \cA^\op = (A^e, w^e)$ where $A^e := A \otimes_k A$
and $w^e := w \otimes 1 - 1 \otimes w$.

A central ingredient in the proof of  Theorem \ref{MainTheorem} is the fully-faithful embedding of $k$-linear dg categories
\begin{equation} \label{E1212}
\psi: \Perfdg(\cA)^e \into \Perfdg(\cA^e)
\end{equation}
given on objects by $\psi(X,Y) = X \otimes_k Y^*$, where $Y^* := \Hom_A(Y, A) \in \Perfdg(\cA^\op)$ is the $A$-linear dual of $Y$. 
The map on hom complexes  is given by the natural isomorphisms
  $$
\begin{aligned}
  \Hom_{\Perfdg(\cA)^e}((X_1, Y_1), (X_2, Y_2)) 
&=  \Hom_A(X_1, X_2) \otimes_k \Hom_A(Y_2, Y_1) \\
&\cong X_1^* \otimes_A X_2 \otimes_k Y_2^* \otimes_A Y_1 \\ 
&\cong (X_1 \otimes_k Y_1^*)^* \otimes_{A^e}  (X^2 \otimes_k Y_2^*) \\
&\cong \Hom_{\Perfdg(\cA^e)}(X_1 \otimes_k Y_1^*, X_2 \otimes_k Y_2^*). \\
\end{aligned}
$$
As we shall see, this functor is not always a Morita equivalence, and this is ultimately the reason why local cohomology enters into the formula for $HH(\Perfdg(\cA))$.

Another important fact is that the dg $\Perfdg(\cA)^e$-modules $\Deltaright$ and $\Deltaleft$ can be extended along $\psi$, as follows:
Observe that, since the multiplication map $A^e \onto A$ sends $w^e$ to $0$,
we may regard $A$ as either a left or a right curved $\cA^e$-module. Define $h^A$ to be the right dg $\Perfdg(\cA^e)$-module given on objects by
$$
h^A(Z) := \Hom_{\cA^e}(Z, A),  \, \text{ for $Z \in \Perfdg(\cA^e)$}
$$
and $t^A$ to be the left dg $\Perfdg(\cA^e)$-module given on object by
$$
t^A(Z) := A \otimes_{\cA^e} Z  \, \text{ for $Z \in \Perfdg(\cA^e)$.}
$$
For $X, Y \in \Perfdg(\cA^e)$ we have isomorphisms
$$
h^A(X \otimes_k Y^*) \cong (X^* \otimes_k Y) \otimes_{\cA^e} A \cong X^* \otimes_\cA Y \cong \Hom_{\Perfdg(\cA)}(X, Y)
$$
 and
$$
t^A(X \otimes_k Y^*) \cong A^\rght \otimes_{\cA^e} (X \otimes_k Y^*) \cong Y^* \otimes_\cA X \cong \Hom_{\Perfdg(\cA)}(Y, X).
$$
The naturality of these isomorphisms gives isomorphisms of dg $\Perfdg(\cA)^e$-modules
\begin{equation} \label{E1215}
  \psi_* h^A  \cong \Deltaleft \wideand
  \psi_* t^A  \cong \Deltaright
\end{equation}
where $\psi_*$ denotes restriction of scalars along $\psi$ (see \ref{ss:dgmodules} below). 
This gives the identification
\begin{equation} \label{E422}
  HH(\Perfdg(\cA))   \cong \psi_* h^A \Lotimes_{\Perfdg(\cA)^e} \psi_* t^A.
 \end{equation}

Now, for all $X, Y \in \Perfdg(\cA)$, the curved $\cA^e$-module $X \otimes_k Y^*$ is supported 
on the subset $W := \Nonreg(\cA)\times_{\Spec(k)} \Nonreg(\cA)$ of $\Spec(A \otimes_k A) = \Spec(A) \times_{\Spec(k)} \Spec(A)$
and thus $\psi$ factors through $\Perfdg^W(\cA^e)$, the full dg subcategory of objects
supported on $W$; see Theorem \ref{thm47}.  That is, we have a commutative triangle 
\begin{equation} \label{E1230x}
\begin{tikzcd}
\Perfdg(\cA)^e \arrow[rr, hook, "\phi"] \arrow[rd, hook, "\psi"]&& \Perfdg^W(\cA^e) \arrow[dl, hook', "\iota"] \\ & \Perfdg(\cA^e)
\end{tikzcd}
\end{equation}
of fully-faithful embeddings of dg categories, where $\iota$ is an inclusion.

One may organize the proof of Theorem \ref{MainTheorem} into three major steps:

The first step is to show that the canonical map
\begin{equation} \label{Isomorphism3}
\iota_* h^A \Lotimes_{\Perfdg^W(\cA^e)} \iota_* t^A \xra{\cong} \R\G_W \left( h^A \Lotimes_{\Perfdg(\cA^e)}  t^A\right)
\end{equation}
is an  isomorphism in the derived category of dg $A$-modules.
This is part of a larger phenomenon concerning the relationship between tensor products of modules over dg categories and full
dg subcategories defined by support conditions; see Theorem \ref{thm711b} for the precise statement.



The second step is to prove
the fully-faithful embedding $\phi$ is a Morita equivalence, and thus the canonical map
\begin{equation} \label{Isomorphism4}
\psi_* h^A   \Lotimes_{\Perfdg(\cA)^e} \psi_* t^A \xra{\cong} \iota_* h^A \Lotimes_{\Perfdg^W(\cA^e)} \iota_* h^A
\end{equation}
is an isomorphism in the derived category of dg $A$-modules.
This step is a consequence of a theorem due to B.~Briggs presented in the appendix to this paper; see also Theorem \ref{thm47} in the body of this paper.

Combining these two steps with \eqref{E422} yields the isomorphism
\begin{equation} \label{E419b}
  HH(\Perfdg(\cA)) \cong \R\G_W \left( h^A \Lotimes_{\Perfdg(\cA^e)}  t^A\right)
\end{equation}
in the derived category of dg $A$-modules. 
The final step is a calculation of      $h^A \Lotimes_{\Perfdg(\cA^e)}  t^A$ in terms of twisted de Rham complexes.
This is carried out, for the general case,  in Section \ref{sec:nochar}. 
Section \ref{sec:char0} concerns the special case 
when $k$ is a $\Q$-algebra, in which case Theorem \ref{MainTheorem} includes assertions regarding naturality and
the relationship between the Connes' $B$ operator and the de Rham differential.
A key ingredient in the proofs of these results is the notion of ``type II'' derived tensor products and ``type II'' Hochschild homology, developed in \cite{Positselski} and \cite{PP}.

Finally, Section \ref{sec:whensmooth} builds on the results established in the proof of Theorem \ref{MainTheorem} to establish Theorem \ref{introsmooththm}, characteristing which
matrix factorization categories are homologically smooth.

\section{Background on homological algebra over dg categories} \label{sec:homalg}

In this section, we collect the needed background
concerning the theory of derived functors of dg modules on dg categories. We start by fixing notation.

\subsection{Notation and terminology} \label{subsec:notation}

Unless otherwise indicated, all gradings are $\Z$-gradings with components indexed using superscripts,  and the differential in a chain complex has degree $1$.
If $M = M^\cdot$ is a graded module, then $\Sigma^n M$ is graded by $(\Sigma^n M)^j := M^{j+n}$. If $M = (M^\cdot, d_M)$ is a complex, then $\Sigma^n M$ is the complex
$(\Sigma^n M^\cdot, (-1)^n d_M)$. 

Throughout the paper we fix a graded commutative ring $k = \bigoplus_i k^i$ that is concentrated in even degrees (i.e., $k^i = 0$ for $i$
odd), and we assume $k$ is regular and excellent  (as an ordinary ring,
upon forgetting the grading). In most applications, $k$ will be $F$, $F[t_1, \dots, t_c]$ or $F[t, t^{-1}]$ for a field $F$ and indeterminants $t, t_1, \dots, t_c$ of degree two.

\subsection{dg categories}  \label{subsec:dgcat}
Recall that a dg $k$-module is a graded $k$-module $M = \bigoplus_i M^i$ equipped with a $k$-linear endomorphism $d$ of  degree $1$ that squares to $0$.
When $k$ is an ordinary ring (i.e., $k^i = 0$ for all $i \ne 0$), this is the same thing as a complex of $k$-modules. But a dg module over $k$ should not be
confused with a complex of graded $k$-modules --- the latter determines a dg module upon totalization, but this functor is far from an equivalence. 
We  write $\Mod(k)$ for the abelian category of all dg $k$-modules, in which morphisms are degree zero  $k$-linear maps that respect the differentials.
This category comes with an internal tensor product pairing, written $- \otimes_k -$, making it a monoidal category. 

On occasion, we will need to consider complexes of dg $k$-module, 
in which case we will use subscripts, as in $\cdots \to M_j \to M_{j-1} \to \cdots$ with each $M_j$ equipped with an ``internal'' grading $M_j = \bigoplus_i M_j^i$
and differential $d_{M_j}: M_j^i \to M_j^{i+1}$. Given such a complex of dg modules, its {\em direct sum totalization} is the dg module with $\bigoplus_{i-j = m} M^i_j$ in degree $m$
and its {\em direct product totalization} is the dg module with $\prod_{i-j = m} M^i_J$ in degree $m$.

A {\em $k$-linear dg category} (also known as a {\em dg category over $k$}) is a category  enriched over the monoidal category $(\Mod(k), \otimes_k)$; that is, a $k$-linear dg category, say $C$, consists
of a collection of objects $\ob(C)$, for each pair $X, Y \in \ob(C)$ a dg $k$-module $\Hom_C(X, Y)$, and for each triple $X, Y , Z \in \ob(C)$
a morphism of dg $k$-modules $\Hom_C(X, Y) \otimes_k \Hom_C(Y, Z) \to \Hom_C(X, Z)$. This structure is required to satisfy the usual unital and associative axioms. 
When $k = \Z$ (concentrated in degree $0$), a $\Z$-linear dg category is just called a {\em dg category}.

For instance,  a $k$-linear dg category with just one object is the same thing as a dga (differential graded algebra) over $k$. 

The category of all chain complexes of abelian groups, with morphisms given by the usual chain complex of maps between two chain complexes, is a dg category.
More generally, for any graded $k$-algebra $A$, the collection of dg $A$-modules forms a $k$-linear dg category, written as $\Moddg(A)$.

For a curved ring $\cA = (A, w)$, the collection of all curved-$\cA$-modules forms a $A$-linear dg category we write as $\Moddg(\cA)$. The hom complex for a pair of curved modules
$M = (M^\nat, d_M)$ and $N =(N^\nat, d_N)$ is given by the exact same formulas for dg $A$-modules; in detail, $\Hom_{\Moddg(\cA)}(M, N)$ is the usual graded $A$-module $\Hom_A(M, N)$
equipped with differential $\del$ given by $\del(g) = d_N \circ g - (-1)^{|g|} g \circ d_M$. Recall that, when $A$ is regular, we write $\Perfdg(\cA)$ for the full dg subcategory
consisting of those curved modules $P$ such that $P^\nat$ is finitely generated and projective as a graded $A$-module.

Given a dg category $C$, we write $Z^0C$ for the ordinary, pre-additive category with the same objects as $C$ and morphisms given by
$\Hom_{Z^0C}(X, Y) = Z^0 \Hom_C(X, Y)$, the abelian group of degree $0$ cycles. The {\em homotopy category} associated to $C$, written $[C]$, is the pre-additive
category with the same objects as $C$ and with $\Hom_{[C]}(X, Y) = H^0 \Hom_C(X, Y)$. In many of the  situations of interest in this paper,  $Z^0C$ will be an exact category and
$[C]$ will be triangulated. For example, the category $Z^0 \Moddg(k)$ may be identified with $\Mod(k)$.

A {\em dg functor} between two $k$-linear dg categories, say from $C$ to $D$,  is a  functor between categories enriched over $\Mod(k)$. So, it consists of
function $F: \ob(C) \to \ob(D)$ on objects and for each pair $X, Y \in \ob(C)$ a morphism of dg $k$-modules $F: \Hom_C(X,Y) \to \Hom_D(F(X), F(Y))$. These maps are required 
to respect units and compositions.

\subsection{dg modules over dg categories} \label{ss:dgmodules}

Given a $k$-linear dg category $C$, a {\em (left) dg $C$-module} $M$ is a dg functor $M: C \to \Moddg(k)$.
So, $M$ consists of a collection of dg $k$-modules $M(X)$ indexed by $X \in \ob(C)$ together with morphisms 
$\Hom_C(X,Y) \otimes_k M(X) \to M(Y)$ of dg $k$-modules for all pairs of objects, such that the evident unital and associative axioms hold.
When $C$ has just one object, and thus may be identified with a dg $k$-algebra $A$, a left dg $C$-module is the same thing as a left dg $A$-module.
Morphisms of dg modules are given by natural transformations of dg functors; when $C$ has just one object, this coincides with the usual notion of
a morphism of dg modules over a dg algebra. In general, the collection of all left dg modules over $C$, written $\Mod(C)$, forms an abelian category. 

For any integer $i$, the {\em $i$-th suspension} of a dg $C$-module $M$, written $\Sigma^i M$, is the composition of $M$ with the $i$-th suspension functor on the category
of dg $k$-modules; i.e, $(\Sigma^i M)(X) = \Sigma^i (M(X))$ for all $X \in C$.
Similarly, if $g: M \to N$ is a morphism in $\Mod(C)$, its {\em mapping cone} (or just {\em cone} for short) is the dg $C$-module defined via composition: $\cone(g)(X) = \cone(g(X): M(X) \to N(X))$.
We have a canonical short exact sequence  $0 \to N \to \cone(g) \to \Sigma M \to 0$ of dg $C$-modules.

Just as $\Mod(k)$ may be enriched to the dg category $\Moddg(k)$, 
we may enrich the collection of all dg $C$-modules to a dg category, which we write as  $\Moddg(C)$. In detail, given dg $C$-modules $M$ and $N$,
$\Hom_{\Moddg(C)}(M, N)$ is the kernel of the map
$$
\prod_{X \in C} \Hom_{\Moddg(k)}(M(X), N(X)) \to \prod_{Y, Z \in C} \Hom_{\Moddg(k)}(  \Hom_C(Y, Z),  \Hom_{\Moddg(k)}(M(Y),N(Z)))
$$
sending a tuple $(\a_X)_{X \in C}$ to the tuple $(g \mapsto g_* \a_Y - \a_Z g_*)_{Y, Z}$, where $\Hom_{\Moddg(k)}$ refers to hom for the dg category of dg $k$-modules.
The categories $\Mod(C)$ and $\Moddg(C)$ are related by $\Mod(C) = Z^0 \Moddg(C)$.

For a fixed dg $C$-module $N$, we may interpret $\Hom_{\Moddg(C)}(-, N)$ as a functor from $\Mod(C)$ to $\Mod(k)$, and 
as such it is contravariantly left exact, and similarly  $\Hom_{\Moddg(C)}(M, -)$ is covariantly left exact, just as in more classical settings.

A {\em graded category} is a dg category $C$ in which $\Hom_C(X, Y)$ has zero differential for all objects $X$ and $Y$, and for such a category,
a {\em graded C-module} is a dg $C$-module in which all differentials are $0$.
For a graded category $C$, we write $\grMod(C)$ for the abelian category of all graded $C$-modules. 

For any chain complex $M$ we write $M^\nat$ for the graded object obtained by forgetting the differential. 
Given a dg category $C$, we write $C^\nat$ for the graded category with the same objects and homs given by $\Hom_{C^\nat}(X, Y) = \Hom_C(X, Y)^\nat$.
If $M$ is a dg $C$-module, then $M^\nat$ is the graded $C^\nat$-module given by $X \mapsto M(X)^\nat$.  From the definition of $\Hom_{\Moddg(C)}$ we have
\begin{equation} \label{E1210}
\Hom_{\Moddg(C)}(M, N)^\nat = \Hom_{\grMod(C^\nat)}(M^\nat, N^\nat).
\end{equation}

For any complex $M$, let $H(M) = \bigoplus_j H^j(M)$ be the graded object given by taking its homology. 
By $H(C)$ we mean the graded category with the same objects as $C$ with morphisms given by
$$
\Hom_{H(C)}(X, Y) = H(\Hom_C(X,Y)). 
$$
Given a (right or left) dg $C$-module $M$, we let $H(M)$ denote the graded $H(C)$-module given by $H(M)(X) = H(M(X))$.

The homotopy category associated to the dg category $\Moddg(C)$, namely $[\Moddg(C)]$,  is a triangulated category with translation functor $M \mapsto \Sigma M$
and with a distinguished triangle being, by definition, a triangle isomorphic to one of the form $M \xra{g} N \xra{\can} \cone(g) \xra{\can} \Sigma M$ for a morphism $g$ of
dg modules.  A {\em quasi-isomorphism} of dg $C$-modules is a morphism $M \to N$ in $\Mod(C)$ that induces a quasi-isomorphism in the usual sense
$M(X) \xra{\sim} N(X)$ for all $X \in C$.  A dg $C$-module $M$ is {\em quasi-trivial} if $M \to 0$ is a quasi-isomorphism. 
The {\em derived category} of a dg category $C$, written $\cD(C)$,  is the Verdier localization of the triangulated category  $[\Moddg(C)]$ given by inverting
all quasi-isomorphisms of dg $C$-modules or, equivalently, modding out by all quasi-trivial dg $C$-modules.

For a dg category $C$, we write $C^\op$ for the opposite dg category, which has  the same objects as $C$ but
its hom complexes are given by $\Hom_{C^\op}(X, Y) = \Hom_C(Y,X)$. A {\em right dg module} over $C$ is by definition a left dg module over $C^\op$.
Equivalently, a right dg $C$-module $N$ consists of a collections of dg $k$-modules $N(X)$ indexed by $X \in \ob(C)$ together with morphisms
$N(Y) \otimes_k \Hom_C(X, Y)  \to N(X)$ of dg $k$-modules that satisfy the usual unital and associative conditions.

If $\phi: D \to C$ is a dg functor between dg categories, and $M$ is a dg $C$-module, by $\phi_*(M)$ we mean the dg $D$-module given by {\em restriction of scalars}
along $\phi$; that is, $\phi_*(M)$ is the composition of dg functors $M \circ \phi$. Explicitly, $\phi_*(M)$ sends an object $X$ of $D$ to the dg $k$-module
$M(\phi(X))$ and the required pairings are given by the composition of 
$$
\Hom_D(X, Y) \otimes_k M(\phi(X)) \xra{\phi \otimes \id} \Hom_C(\phi(X), \phi(Y)) \otimes M(\phi(X)) \to M(\phi(Y)).
$$

 \subsection{Various classes of dg $C$-modules}  \label{subsec:classes}

Throughout this section, $C$ denotes a $k$-linear dg category.  We review and develop some basic notions concerning the  homological algebra of dg $C$-modules.
Many of the results presented here are drawn from \cite{Drinfeld}, \cite{Keller}, \cite{LuntsOrlov}, \cite{Orlov}, \cite{Positselski}, \cite{PP} and \cite{Toen}.
Proofs are included only for results that are not easily found in these references.

For $X \in C$, define $h_X$ to be the left dg $C$-module
$$
h_X := \Hom_C(X, -)
$$
and $h^X$ to be the right dg $C$-module
$$
h^X := \Hom_C(-, X).
$$
A left (resp. right) dg module is {\em representable} if it is isomorphic to $h_X$ (resp. $h^X)$ for some object $X$.
For any object $X$ of $C$ and any left dg $C$-module $M$, we have an isomorphism of dg $k$-modules
\begin{equation} \label{E1210b}
\Hom_{\Moddg(C)}(h_X, M) \cong M(X)
\end{equation}
and similarly for any right dg $C$-module $N$ we have
\begin{equation} \label{E1210y}
\Hom_{\Moddg(C^\op)}(h^X, N) \cong N(X)
\end{equation}

 A module is {\em free} (resp., {\em finite free}) if it is a direct
 sum (resp., finite direct sum) of suspensions of representable modules: $\bigoplus_i \Sigma^{j_i} h_{X_i}$ or $\bigoplus_i \Sigma^{j_i} h^{X_i}$.
 A dg $C$-module $F$ is {\em semi-free} if it admits a {\em semi-free filtration}, which is defined to be a filtration
$0 = F_{-1} \subseteq F_0 \subseteq F_1 \subseteq \cdots \subseteq F$ by submodules such that 
$F = \bigcup_i F_i$ and $F_i/F_{i-1}$ is free for all $i$.
All free modules are semi-free and the collection of semi-free modules is closed under suspension and arbitrary direct sums.

\begin{prop} \label{newprop47} \cite[\S 3.1]{Keller} For any dg $C$-module $M$, there exist a quasi-isomorphism $F \xra{\sim} M$ with $F$ semi-free.
\end{prop}

A dg $C$-module $P$ is called {\em homotopically projective}, or {\em h-projective} for short, if $\Hom_{\Moddg(C)}(P, T)$ is exact for all quasi-trivial
dg $C$-modules $T$. 
Equivalently, $P$ is h-projective if $\Hom_{\Moddg(C)}(P, -)$ sends quasi-isomorphisms of dg $C$-modules to quasi-isomorphisms of chain complexes. 
The collection of h-projective modules is closed under suspension, arbitrary co-products, and summands in the abelian category $\Mod(X)$. 
It is also closed under homotopy equivalence; for instance, all contractible dg $C$-modules are h-projective.  
A dg $C$-module $P$ is {\em graded projective} if the functor $\Hom_{\Moddg(C)}(P, -)$ sends surjections in $\Mod(C)$ to surjections
of chain complexes. This condition depends only on the graded $C^\nat$-module $P^\nat$.
A dg $C$-module $P$ is {\em semi-projective} if it is both h-projective and graded projective.
The collection of semi-projective modules is closed under suspension, arbitrary co-products, and summands in the abelian category $\Mod(X)$. 
(It is not closed under homotopy equivalence.)

All semi-free modules are semi-projective.
Writing $\SemiFreedg(C)$, $\SemiProjdg(C)$, and $\HProjdg(C)$ for the full dg subcategories of $\Moddg(C)$ consisting of the semi-free, the semi-projective, and
the h-projective objects, respectively, we have containments
$$
\SemiFreedg(C) \subseteq \SemiProjdg(C) \subseteq \HProjdg(C) \subseteq \Moddg(C).
$$

Recall that $\cD(C)$ is the Verdier localization of $[\Moddg(C)]$ given by inverting all quasi-isomorphisms.
A morphism of h-projective modules is a quasi-isomorphism if and only if it is a homotopy equivalence, and thus the canonical functors
$[\SemiFreedg(C)] \into [\SemiProjdg(C)] \into [\HProjdg(C)] \into \cD(C)$ are fully-faithful.
An by Proposition \ref{newprop47}  each is essentially onto and thus we have equivalences
\begin{equation} \label{E1212t}
 [\SemiFreedg(C)] \xra{\cong} [\SemiProjdg(C)] \xra{\cong} [\HProjdg(C)] \xra{\cong} \cD(C)
 \end{equation}
 of triangulated categories. We will often use $[\SemiFreedg(C)]$ in place of $\cD(C)$.

 \subsection{Yoneda embedding and pre-triangulated dg categories}  \label{subsec:Yoneda}
 The {\em dg Yoneda embedding} is the fully-faithful dg functor from $C$ to the dg category of right modules over $C$, 
   $$
   \Yon_\dg: C \to \SemiFreedg(C^\op) \subseteq \Moddg(C^\op)
   $$
   that is  given on objects by $\Yon_\dg(X) = h^X$ and on morphisms by the canonical isomorphism $\Hom_C(X, Y) = h^Y(X) \cong \Hom_{\Moddg(C^\op)}(h^X, h^Y)$.

 The dg Yoneda embedding induces a fully-faithful functor on the associated additive categories given by taking $0$ cycles
 $$
 \Yon: Z^0 C \into Z^0 \SemiFreedg(C^\op) 
 $$
 and a fully-faithful functor on homotopy categories
 $$
 [\Yon]: [C] \into [\SemiFreedg(C^\op)] \cong \cD(C^\op).
 $$

 The category $Z^0 \SemiFreedg(C^\op)$ comes equipped with suspensions of objects  and mapping cones of  morphisms. 
 We say $C$ is {\em pre-triangulated} if $Z^0 C$ admits suspensions and mapping cones too, that are compatible with the Yoneda embedding; that is,
 $C$ is pre-triangulated if for each integer $j$, there is a functor $\Sigma^j: Z^0C \to Z^0C$ such that $\Sigma^j \circ \Yon$ is naturally isomorphic
 to $\Yon \circ \Sigma ^j$, and for each morphism $g: X \to Y$ in $Z^0 C$, there is an object $\cone(g)$ of $C$ and morphisms $Y \to \cone(g) \to \Sigma(X)$
 such that $\Yon(Y) \to \Yon(\cone(g)) \to \Yon(\Sigma(X))$ is naturally isomorphic to $\Yon(Y) \to \cone(\Yon(g)) \to \Sigma(\Yon(X))$.
 The assignment $g \mapsto C(g)$ for morphisms in $Z^0C$ is required to be functorial, in the sense that given a commutative square
 $$
 \begin{tikzcd}
 X \ar[r]^g \ar[d, "\a"] & Y \ar[d, "\b"] \\
 X' \ar[r, "g"] & Y' \\
 \end{tikzcd}
 $$
in $Z^0 C$, the maps $(\a, \b)$ induce a map from $\cone(g)$ to $\cone(g')$, compatibly via the Yoneda embedding with the usual construction in $\Mod(C^\op)$. 
 In particular, given such a square we define its {\em totalization} as $\cone(\cone(g) \xra{(\a, \b)} \cone(g'))$. 
 As a special case of this, if $X \xra{g} Y  \xra{h} Z$ are morphisms in $Z^0C$ with $h \circ g = 0$, 
 then taking $X' = 0$ and $Y' = Z$, we define  
 $$
 \Tot(X \xra{g} Y  \xra{h} Z)  = \cone(\cone(g) \xra{(0,h)}\cone(0 \to Z))
 $$
 and we have
 $$
 \Yon\left(\Tot(X \xra{g} Y  \xra{h} Z)\right)   \cong 
 \Tot(\Yon(X) \xra{\Yon(g)} \Yon(Y)  \xra{\Yon(h)} \Yon(Z)).
 $$
 More generally, given a bounded complex 
 $$
 X_\cdot = (\cdots \to 0 \to X_m \to \cdots \to X_n \to 0 \to \cdots)
 $$
 in $Z^0(C)$, we may form its totalization in $C$ in a recursive manner, so that 
 $$
 \Yon(\Tot(X_\cdot)) \cong \Tot(\Yon(X_\cdot)). 
 $$

 For example, for a curved ring $\cA$, the dg category $\Moddg(\cA)$ is pre-triangulated, with the notions of suspension and mapping cones
 given in exactly the same was as for dg $A$-modules, and more generally the totalization of a bounded complex of objects in $\Mod(\cA)$ extends
 the usual construction for dg $A$-modules.

\subsection{Tensor products and derived tensor products} \label{ss:tensor}
 Given a right  dg $C$-module $M$ and a left dg $C$-module $N$, their {\em tensor product}, written $M \otimes_C N$, is the dg $k$-module given by the cokernel of
the map
$$
\bigoplus_{Y, Z \in C} M(Z) \otimes_k \Hom_C(Y,Z) \otimes_k N(Y) \to 
\bigoplus_{X \in C} M(X) \otimes_k N(X)
$$
that sends $m \otimes g \otimes n$ to $m \cdot g \otimes n - m \otimes g \cdot n$. 
 If $C$ has just one object whose endomorphism dga is $R$, then $M \otimes_C N$ coincides with the usual tensor product of dg $R$-modules.
 The tensor product is made into a bifunctor, from $\Mod(C^\op) \times \Mod(C)$ to $\Mod(k)$, in the evident manner.
 For a fixed $M$ and $N$, the functors $M \otimes_C -$ and $- \otimes_C N$ are both right-exact functors that commute with taking suspensions.

 From the defintion, we observe that the graded $k$-module underlying $M \otimes_C N$ does not depend on the differentials; that is,
 \begin{equation} \label{E31}
   (M \otimes_C N)^\nat = M^\nat \otimes_{C^\nat} N^\nat.
  \end{equation}

 Given an object $X \in C$, there are natural isomorphisms
 \begin{equation} \label{E1029a}
 h^X \otimes_C N \cong N(X)
 \and
 M \otimes_X h_X \cong M(X).
 \end{equation}
 These isomorphisms generalize the classical formulas $R \otimes_R N \cong N$ and $M \otimes_R R \cong M$ for a dg modules over a dga $R$.

 The tensor product of dg $C$-modules preserves arbitrary coproducts:
 Given a collection of left dg $C$-modules $\{N_i\}_{i \in I}$ we have a natural isomorphism
 \begin{equation} \label{E26a}
 M \otimes_C \bigoplus_i N_i \cong \bigoplus_i (M \otimes_ C N_i) 
 \end{equation}
 for all right dg $C$-modules $M$, and similarly in the other argument.

\begin{defn} The {\em derived tensor product} of a right dg $C$-module $M$ and a left dg $C$-module $N$, written
  $M \Lotimes_C N$, is the dg $k$-module 
  $F \otimes_C N$ where $F \xra{\sim} M$ is a chosen quasi-isomorphism such that $F$ is semi-free; see Proposition \ref{newprop47}.
\end{defn}

The usual arguments in the classical setting show that $M \Lotimes_C N$ is independent, up to isomorphism  in $\cD(k)$, of the semi-free resolution of $M$ chosen.
Moreover, we have $M \Lotimes_C N \sim M \otimes_C F'$ if $F' \xra{\sim} N$ is a quasi-isomorphism with $F'$ semi-free.

 \subsection{h-flat and semi-flat modules} \label{subsec:flat}
 A dg $C$-module $E$ is {\em homotopically flat}, or {\em h-flat} for short, if the chain complex $M \otimes_C E$ is exact
 whenever $M$ is a quasi-trivial right dg $C$-module. This is equivalent to the condition that $- \otimes_C E$ preserves quasi-isomorphisms. 
We say $E$ is {\em graded-flat} if $- \otimes_C E$ is an exact functor from $\Mod(C)$ to $\Mod(k)$. Observe that by \eqref{E31}, the graded-flat property
does not depend on differentials; that is, $E$ is a graded-flat dg $C$-module if and only if $E^\nat$ is flat as a graded $C^\nat$-module. 
Finally $E$ is {\em semi-flat} if it is both h-flat and graded-flat. 
It can be shown that $E$ is semi-flat if and only if whenever $N' \xira{\sim} N$ is an injective quasi-isomorphism of right dg $C$-modules,
$N' \otimes_C M \xira{\sim} N \otimes_C M$ is an injective quasi-isomorphism of dg $k$-modules. We extend these notions to right modules in the evident way.

\begin{prop} \label{prop1111} The collection of h-flat dg $C$-modules
 is closed under suspension, summands, and filtered colimits taken in the abelian category $\Mod(C)$.
 The totalization of any bounded below complex of h-flat dg $C$-modules is again h-flat.

The collection of semi-flat modules is also closed under suspension, summands and filtered colimits. 
The totalization of a complex of h-flat dg $C$-modules that is bounded (both above and below) is again h-flat.
This collection is also closed under extension in $\Mod(C)$.

All semi-free dg $C$-modules are semi-flat.
\end{prop}

\begin{prop} \label{hflatcor} For a right dg $C$-module $M$ and a left dg $C$-module $N$, given a quasi-isomorphism $G \xra{\sim} N$ such that $G$ is semi-flat 
we have an isomorphism
 $$
 M \otimes_C G \xra{\cong}  M \Lotimes_C N
 $$
 in $\cD(k)$, and similarly for semi-flat resolutions of $M$.
 \end{prop}

 \subsection{Extension of scalars, Morita equivalences} \label{subsec:Morita}
 Let $\phi: D \to C$ be a dg functor between dg categories and recall that $\phi_*$ denotes restriction
of scalars of dg modules along $\phi$.
Given a (right or left) dg $D$-module $M$, we write $\phi^*(M)$ for the dg $C$-module obtained by {\em extension of scalars} along $\phi$.
When $M$ is a right module, $\phi^*(M)$ is given on objects by 
 $$
 \phi^*(M)(X) =  M \otimes_D \phi_*(h^C_X), \text{ for all $X \in C$,}
 $$
 and similarly for left modules.

 By construction, $(\phi^*, \phi_*)$ form an adjoint pair: there is a natural isomorphism
 $\Hom_{\Moddg(C)} (\phi^*(M), N) \cong \Hom_{\Moddg(D)} (M, \phi_*(N))$ for any dg $D$-module $M$ and dg $C$-module $N$.
 In particular, there is a natural transformation from the identity functor on $\Moddg(M)$ to $\phi_* \circ \phi^*$ and
 a natural transformation from $\phi^* \circ \phi_*$ to the identity functor on
 $\Moddg(C)$.

 If $M = h_D^Y$ for $Y \in D$,  then we have $h^Y_D \otimes_D \phi_*(h^C_X) \cong \phi_*(h^C_X)(Y) = \Hom_C(X,\phi(Y))$,
 and thus there is a natural isomorphism
 $$
 \phi^*(h_D^Y) \cong h_C^{\phi(Y)}.
 $$
 This gives a commutative square of dg functors
 \begin{equation} \label{E1230}
 \begin{tikzcd}
   D \arrow[r, "\phi"] \arrow[d, hook, "\Yon^D_\dg"] & C \arrow[d, hook, "\Yon^C_\dg"] \\
   \Moddg(D^\op) \arrow[r, "\phi^*"] & \Moddg(C^\op)
 \end{tikzcd}
 \end{equation}
 in which the vertical maps of the dg Yoneda embeddings.

 Unlike for restriction of scalars, extension of scalars fails, in general, to preserve quasi-isomorphisms, and we will have need of its derived version. 
 If $F$ is a semi-free dg $D$-module, then $\phi^*(F)$ is a semi-free dg $C$-module, and we thus have an induced commutative square
 \begin{equation} \label{E1230b}
 \begin{tikzcd}
   D \arrow[r, "\phi"] \arrow[d, hook, "\Yon^D_\dg"] & C \arrow[d, hook, "\Yon^C_\dg"] \\
   \SemiFreedg(D^\op) \arrow[r, "\phi^*"] & \SemiFreedg(C^\op)
 \end{tikzcd}
 \end{equation}
 of dg categories.
 We define $\L \phi^*: \cD(D^\op) \to \cD(C^\op)$, {\em derived extension of scalars along $\phi$}, as the composition
 of
 $$
 \cD(D^\op) \cong [\SemiFreedg(D^\op)] \xra{\phi_*} [\SemiFreedg(C^\op)] \cong \cD(C^\op).
 $$
 So, we have commutative square of triangulated categories
 \begin{equation} \label{E1230x}
 \begin{tikzcd}
   {[D]} \arrow[r, "\phi"] \arrow[d, hook,"{[\Yon^D]}"] & {[C]} \arrow[d, hook, "{[\Yon^C]}"] \\
   \cD(D^\op) \arrow[r, "\L\phi^*"] & \cD(C^\op).
 \end{tikzcd}
 \end{equation}
 The functors $(\L\phi^*: \cD(D) \to \cD(C), \phi_* \cD(C) \to \cD(D))$ form an adjoint pair.
 Moreover, we have:

 \begin{prop} \label{prop1230c} For a right dg $D$-module $M$  and a left dg $C$-module $N$, there is a natural isomorphism
   $$
   \phi^*(M) \otimes_C N \cong M \otimes_D \phi_*(N)
   $$
   in $\Mod(k)$ and a natural isomorphism
   $$
   \L \phi^*(M) \Lotimes_C N \cong M \Lotimes_D \phi_*(N)
   $$
   in $\cD(k)$. 
 \end{prop}

 \begin{proof}
   Fixing $M$, let $\cS_M$ be the collection of those dg $C$-modules $N$ for which the canonical map  $\phi^*(M) \otimes_C N \to M \otimes_D \phi_*(N)$ is an isomorphism.
If $N = h_X^C$ then $\phi^*(M) \otimes_C h_X^C \cong \phi^*(M)(X) \cong M \otimes_D \phi_*(h_X^C)$, and thus $h_X^C \in \cS_M$ for all $X$.
Since this is a natural transformation of functors from $\Mod(C)$ to $\Mod(k)$ that preserve all coproducts, 
$\cS_M$ contains all free modules. Since the two functors are left-exact, it follows that $\cS_M$ contains all dg $C$-modules.

For the second assertion,   we may assume $M$ is semi-free, in which case $\L \phi^*(M) = \phi^*(M)$ is also semi-free, and thus the second isomorphism follows from the first. 
\end{proof}

 \begin{defn} A dg functor $\phi: D \to C$ is a {\em Morita equivalence} if the induced functor
   $\phi^*: \SemiFreedg(D) \to \SemiFreedg(C)$ is a quasi-equivalence or, equivalently,
 $\L \phi^*: \cD(D) \to \cD(C)$ is an equivalence of triangulated categories.
 \end{defn}

The square \eqref{E1230b} shows that if $\phi$ is a Morita equivalence, then $\phi$ must be quasi-fully-faithful.
 Since $\L \phi^*$ is adjoint to $\phi_*: \cD(C) \to \cD(D)$, if $\phi$ is a Morita equivalence,
 then $\phi_*$ is also an equivalence, and it is the inverse up to natural isomorphism of $\L \phi^*$.

 Every quasi-equivalence is a Morita equivalence, and this generalizes to the following result. Recall that the {\em thick closure} of
 a collection of objects $S$ in a triangulated category $\cT$, written $\Thick_{\cT}(S)$, is the smallest triangulated subcategory of $\cT$ that is closed under taking summands.

 \begin{prop} \label{thickprop}
   \cite[1.15]{LuntsOrlov} Suppose $\phi:D \into C$ is a quasi-fully-faithful dg functor and $C$ is pre-triangulated. 
   If $\Thick_{[C]}(\{\phi(X) \mid X \in D\}) = [C]$, then $\phi$ is a Morita equivalence. 
 \end{prop}

The following will play an important role in establishing isomorphism \eqref{Isomorphism4}. 

 \begin{prop} \label{Moritaprop}
   Assume $\phi: D \to C$ is a Morita equivalence, $M$ is a right dg $C$-module and $N$ is a left dg $C$-module.
   Then the canonical map
   $$
   \phi_*(M) \Lotimes_D \phi_*(N) \xra{\cong} M \Lotimes_C N
   $$
   is a an isomorphism in the derived category of dg $k$-modules. 
 \end{prop}

 \begin{proof} By Proposition \ref{prop1230c} we have
   $\phi_*(M) \Lotimes_D \phi_*(N) \cong \L\phi^* \phi_* M \Lotimes_C N$  and if $\phi$ is a Morita equivalence, then the canonical map $\L\phi^* \phi_* M \to M$ is an isomorphism
   in $\cD(C)$, so that  $\L\phi^* \phi_* M \Lotimes_C N  \cong M \Lotimes_C N$.
   \end{proof}

   \begin{ex} \label{ex1212}
   Let $R$ be an ordinary commutative ring and let $\Perfdg(R)$, the dg category of all bounded complexes of finitely generated projective $R$-modules.
   We may regard $R$ as the full dg subcategory of $\Perfdg(R)$ on the one object $R \in \Perfdg(R)$, and we  write $\iota: R \into \Perfdg(R)$ for the inclusion functor.

   Given a pair of $R$-modules $M$ and $N$, we may extend them to dg $\Perfdg(R)$-modules $\iota^* M$ and $\iota^* N$ via extension of scalars along $\iota$. (These extensions
   coincide with their derived versions.) 
   Explicitly, $\iota^*{N}$ is the left dg $\Perfdg(R)$-module $P \mapsto P \otimes_R N$ and  $\iota^*{M}$ is the right dg $\Perfdg(R)$-module $P \mapsto M \otimes_R P^*$.
   Note that $\iota_* \iota^*{M}$ is the $R$-module $M \otimes_R R^* \cong M$; i.e., $\iota_* \iota^*{M} \cong M$.
   Similarly, $\iota_* \iota^*{N} \cong N$.
  
   Since $\Perfdg(R)$ is pre-triangulated and the thick closure of $R$ in $[\Perfdg(R)]$ is all of $[\Perfdg(R)]$, by Propositions \ref{thickprop} and \ref{Moritaprop}, we have
   $$
   \iota^*{M} \Lotimes_{\Perfdg(R)} \iota^*{N}  \cong M \Lotimes_R N.
   $$
 \end{ex}

  \section{Derived tensor product for dg categories with supports} \label{sec:thirdiso}

  In this section we prove the theorem below. 
  As an immediate consequence, the  map \eqref{Isomorphism3} is an isomorphism.
  (For the latter, our assumptions that $k$ is excellent and $A$ is essentially of finite type over $k$ 
ensures that $W$ is a Zariski closed subset.)

\begin{thm}  \label{thm711b} Let $A$ be a graded commutative ring concentrated in even degrees and assume that $C$ is a pre-triangulated, $A$-linear dg category.
  Let $Z = V(I) \subseteq \Spec(A)$ for a finitely generated homogeneous ideal $I \subseteq A$,
  let $C^Z$ denote the full dg subcategory on the collection of objects supported on $Z$ and write $\iota: C^Z \into C$ for the inclusion functor. 
For any right dg $C$-module $M$ and left dg $C$-module $N$,  the canonical map
$$
\iota_* M \Lotimes_{C^Z} \iota_* N \xra{\cong} \R\G_Z(M \Lotimes_C N)
$$
is an isomorphism in $\cD(A)$. 
\end{thm}

  Let us explain the terminology in this theorem: For an object $X$ of $C$, we define its {\em support over $A$}, written $\supp_A(X)$,  to be the support of the graded $A$-algebra $H\End_C(X)$:
$$
\supp_A(X) := \{\fp \in \Spec(A) \mid H \End_C(X)_\fp \ne 0\} \subseteq \Spec(A).
$$
So, for a Zariski closed subset $Z$ of $\Spec(A)$, $C^Z$ is the full subcategory of $C$ consisting of those objects $X$ 
such that $\End_C(X)_\fp$ is exact whenever $\fp \notin Z$.
For any pair of objects $X$ and $Y$ of $C$, if at least one belongs to $C^Z$,
then the dg $A$-module $\Hom_C(X, Y)$ is also supported on $Z$.
This holds since the action of $A$ on $H^* \Hom_C(X,Y)$ factors through  both $H^*\End_C(X)$  and $H^*\End_C(Y)$. 
Since $C$ is $A$-linear, so is $C^Z$, and both $M \Lotimes_C N$ and $\iota_* M \Lotimes_{C^Z} \iota_* N$ 
have the structure of dg $A$-modules. The latter is supported on $Z$ and hence the canonical map factors through
$\R\G_Z(M \Lotimes_C N)$, the local cohomology of $M \Lotimes_C N$, which since $I$ is finitely generated may be defined using the \Cech complex;
 the theorem asserts this map is a quasi-isomorphism.
    
Let us illustrate the theorem with a concrete example:

\begin{ex} \label{ex1212b}
  Let $R$ be an ordinary commutative ring and let $\Perfdg(R)$ be the dg category of all bounded complexes of finitely generated projective $R$-modules.
  Recall from Example \ref{ex1212} that given $R$-modules $M$ and $N$, we have an isomorphism
  \begin{equation} \label{E1212c}
  \cM \Lotimes_{\Perfdg(R)} \cN  \cong M \Lotimes_R N
  \end{equation}
  in the derived category,   where we define $\cM = j^* M$ and $\cN = j^* N$,  with
  $j$ being  the inclusion of the dg category with one object $\{R\}$ into $\Perfdg(R)$. Explicitly, $\cM$ and $\cN$ are the dg $\Perfdg(R)$-modules given by $\cM(P) = M \otimes_R \Hom_R(P,R)$ and $\cN(P) = P
  \otimes_R N$ for $P \in
  \Perfdg(R)$.
  
  Now let $Z$ be a Zariski closed subset of $\Spec(R)$ associated to a finitely generated ideal,
  and recall $\Perfdg^Z(R)$ denotes the full dg subcategory consisting of complexes supported on $Z$
  and write $i: \Perfdg^Z(R) \into \Perfdg(R)$ for the inclusion functor.
Theorem \ref{thm711b} along with \eqref{E1212c} gives an isomorphism
  $$
    i_* \cM \Lotimes_{\Perfdg^Z(R)} i_*\cN  \cong \R\Gamma_Z \left(M \Lotimes_R N\right).
    $$
\end{ex}

\subsection{Koszul complexes and local cohomology} 
Before proving the Theorem, we develop the notion of local cohomology for dg categories.

Let $Z = V(I)$ where $I$ is a homogeneous ideal of $A$ generated by a finite list of homogeneous elements $\ug = (g_1, \dots, g_c)$ and let $e_i = \deg(g_i)$. 
    Recall that the Koszul complex on $\ug$ is the dg $A$-module 
$$
\Kos(\ug) := \bigotimes_{j=1}^c (\cone(A[-e_i] \xra{g_i} A)) = \Tot(0 \to  A[-e_1-\cdots -e_c] \to \cdots \to \bigoplus_i A[-e_i] \to A \to 0)
$$
with $A$ in degree $0$. 
    We will also be interested in the $A$-linear dual of $\Kos(\ug)$, written $\Kos^*(\ug)$, which is given explicitly as
    $$
    \Kos^*(\ug) := \bigotimes_{j=1}^c \cone(A \xra{-g_i} A[e_i]) = \Tot(0 \to A \to \bigoplus_i A[e_i] \to \cdots \to A[e_1 + \cdots + e_c] \to 0).
    $$
    with $A$  in degree $0$.   

    For any left (resp., right) dg $A$-module $L$, set $\Kos^*(\ug; L) := \Kos^*(\ug) \otimes_A L$ (resp. $L \otimes_A \Kos^*(\ug)$)); it is  a dg $A$-module supported on $V(I)$
    and so $\Kos^*(\ug; -)$ is a dg functor from dg $A$-modules to dg $A$-modules supported on $V(I)$.

    For a (right or left) dg $C$-module $M$, by $\Kos^*(\ug; M)$ we mean the dg $C$-module given by the composition of
    $$
C \xra{M}     \Moddg(A) \xra{\Kos^*(\ug; -)} \Moddg^Z(A) \subseteq \Moddg(A);
$$
that is, $\Kos^*(\ug; M)(X) := \Kos^*(\ug, M(X))$ for all objects $X$ of $C$. Thus $\Kos^*(\ug; -)$ is a functor from dg $C$-modules to dg $C$-modules supported on $Z$.

Since we assume $C$ is pre-triangulated, there exist cones of morphisms and, more generally, totalizations for bounded complexes in $Z^0 C$,
and thus we can make sense of $\Kos^*(\ug; X)$ as an object of $C$ for $X \in C$. That is, we define
$$
\Kos^*(\ug; X) := \Tot(0 \to X \to \bigoplus_i X[e_i] \to \cdots \to X[e_1+ \cdots +e_c] \to 0).
$$
with $X$ in degree $0$ and the maps being the same as in the defintion of $\Kos^*(\ug)$. 
Observe that if $N$ is any left dg $C$-module, we have a natural isomorphism
\begin{equation} \label{E0101}
  N(\Kos^*(\ug; X)) \cong \Kos^*(\ug, N(X)).
\end{equation}
Applying this when $N = h_Y$ for $Y \in C$
gives
$$
    \Hom_C(Y, \Kos^*(\ug; X)) = \Kos^*(\ug; \Hom_C(Y, X))
$$
    and thus we have an identification of right dg $C$-modules
    \begin{equation} \label{E323}
    h^{\Kos^*(\ug; X)} \cong \Kos^*(\ug; h^X).
  \end{equation}  
  In particular, \eqref{E323} gives that $\End_C(\Kos^*(\ug;X)) \cong \Kos^*(\ug; \End_C(X))$ and hence $\Kos^*(\ug; X) \in C^Z$. 
We will regard $\Kos^*(\ug; -)$ as a dg functor from $C$ to $C^Z$. By construction, the square
$$
\begin{tikzcd}
  C \arrow[rr,"\Kos^*(\ug; -)"]  \arrow[d,"\Yon_{\dg}"] &&   C^Z  \arrow[d,"\Yon_{\dg}"] \\
  \Moddg(C^\op) \arrow[rr, "\iota_* \circ \Kos^*(\ug; -)"] & & \Moddg((C^Z)^\op) \\
\end{tikzcd}
$$
commutes, where $\Yon_{\dg}$ is  dg Yoneda embedding; see \S \ref{subsec:Yoneda}.


Similarly, for any dg $C$-module $N$ we define the dg $C$-module $\Kos(\ug; N)$ supported on $Z$ by
the rule $X \mapsto \Kos(\ug; N(X))$, and given $X \in C$, we define $\Kos(\ug; X)$ as an object of $C$ by taking totalization, as above. The same reasoning gives
an isomorphism of dg $C$-modules
     \begin{equation} \label{E323var2}
    h_{\Kos(\ug; X)} \cong \Kos^*(\ug; h_X).
  \end{equation}


Related to $\Kos^*$ we have the augmented \Cech complex associated to $\ug$. This is dg $A$-module given as the totalization of the complex of graded $A$-modules of the form
$$
\begin{aligned}
\CechC(\ug) & :=
\bigotimes_{j=1}^n \cone(A \xra{\can} A[1/g_j]) \\
& = \left(0 \to A \to \bigoplus_i A[1/g_i] \to \bigoplus_{i <j} A[1/g_i, 1/g_j] \to \cdots \to A[1/g_1, \dots, 1/g_c] \to 0
\right)
\end{aligned}
$$
with $A$ in degree $0$.  Let $\ug^l = (g_1^l, \dots, g_c^l)$. We have canonical map $\Kos^*(\ug^l) \to \Kos^*(\ug^{l+1})$ given as the tensor product of the maps from 
$\cone(A \xra{g^l_i} A[le_i])$ to $\cone(A \xra{g^{l+1}_i} A[(l+1)e_i])$, for $1 \leq i \leq c$,  given by the identity on $A$ and multiplication by $g_i$ on $A[le_i]$.  
Using these maps we realize the \Cech complex as a colimit of Koszul complexes: we have an isomorphism of dg $A$-modules
    $$
    \CechC(\ug) \cong \colim(\Kos^*(\ug) \to \Kos^*(\ug^2) \to \Kos^*(\ug^3) \to \cdots).
    $$

    For a graded left $A$-module $L$ we define
    $$
    \R\G_Z (L) := \CechC(\ug) \otimes_A L. 
    $$
    For each such $L$, $\R\G_Z(L)[1/g_i]$ is exact for all $i$ and hence $\R\G_Z(L)$  is supported on $Z$. Note that we have a natural  map
    $$
    \R\G_Z(L) \to L.
    $$
    If $L$ is a dg $A$-module its differential induces a differential on $\R\G_Z(L)$, and we may thus regard it as a dg $A$-module.
    We have a natural isomorphism
    $$
    H(\R\G_Z(L)) \cong \R\G_Z(H(L)),
    $$
    and,  in particular, the canonical map $\R\G_Z(L) \to L$ is a quasi-isomorphism if any only $L$ is supported on $Z$. 
    $\R\G_Z$ may be regarded, in fact, as  a dg functor from $\Moddg(A)$ to $\Moddg(A)$,  and by the comments just made its image is contained in $\Moddg^Z(A)$, the full dg subcategory
    of dg $A$-modules supported on $Z$.

    For a (right or left) dg $C$-module $M$, we define $\R\G_Z(M)$ to be the dg $C$-module given by the composition of dg functors
    $$
    \Moddg(C) \xra{M} \Moddg(A) \xra{\R\G_Z} \Moddg(A).
    $$
    So, for each object $X$ of $C$ we have $\R\G_Z(M)(X) = \R\G_Z(M(X))$.
    Observe that there is a canonical map
    $$
    \R\G_Z(M) \to M,
    $$
    and it is a quasi-isomorphism of dg $C$-modules if and only if $M$ is supported on $Z$ in the sense that
    $M(X)$ is supported on $Z$ for all objects $X$ of $C$.

   \begin{lem} \label{lem323e}
     For any dg $C$-module $M$, there is a natural isomorphism
     $$
     \R\G_Z(M) \cong \colim(\Kos^*(\ug; M) \to \Kos^*(\ug^2; M) \to \cdots)
     $$
     of dg $C$-modules.
   \end{lem}

   \begin{proof} This holds since there is a natural isomorphism $\colim(L \xra{g} L \xra{g} \cdots) \cong L[1/g]$ for any graded $A$-module $L$.
   \end{proof}

   \begin{rem}
     Since we do not assume $Z^0 C$ is closed under colimits, we cannot in general make sense of $\R\G_Z(X)$ as an object of $C$ for $X \in C$.
     This would be possible for the case $C = \Moddg(\cA)$ of primary interest, since we can take colimits in $\Mod(\cA)$.
But such a construction is not needed to prove Theorem \ref{thm711b}.
   \end{rem}

      \subsection{Proof of Theorem \ref{thm711b}}

We need a couple more lemmas:

  \begin{lem} \label{lem323a} For any dg $C$-module $M$, the canonical map $\iota_* \R\G_Z(M) \to \iota_* M$ is a quasi-isomorphism of $C^Z$-modules.
    \end{lem}
   
    \begin{proof} We need to show that for each $X \in C^Z$ the canonical map $\R\G_Z(M)(X) = \R\G_Z M(X) \to M(X)$ is a quasi-isomorphism, which
      is equivalent to the assertion that the dg $A$-module $M(X)$ is supported on $Z$. This holds since $M(X)$ is a dg $\End_C(X)$-module
      and as an dg $A$-module, $\End_C(X)$ is supported on $Z$.
    \end{proof}

\begin{lem} \label{lem1214}
  If $F$ is a semi-free dg $C$-module, then $\iota_* \Gamma_Z(F)$ is a semi-flat dg $C^Z$-module.
   \end{lem}
    
   \begin{proof}
     We will use throughout that all semi-free modules are semi-flat and that  the collection of semi-flat modules is closed under extensions and filtered colimits;
     see Proposition \ref{prop1111}.
     
Given any object $X$ of $C$ and integer $l \geq 0$,      by \eqref{E323var2}  the dg $C$-module  $\Kos^*(\ug^l, h_X)$ is  represented by the object $\Kos^*(\ug^l, X)$, which belongs
     to $C^Z$.  In particular, the dg $C^Z$-module $\iota_* \Kos^*(\ug^l, h_X)$ is representable and hence semi-flat.
     By Lemma \ref{lem323e} we have $\iota_* \R\G_Z h_X = \colim_l \iota_* \Kos^*(\ug^l, h_X)$,
and hence $\iota_* \R\G_Z(h_X)$ is semi-flat for any $X$.
It follows that $\iota_* \Gamma_Z(F)$ is a semi-flat dg $C^Z$-module whenever $F$ is free.

     Now suppose $F$ is semi-free, so that there is a chain
     $0 = F_{-1} \subseteq F_0 \subseteq \cdots \subseteq F$ with $F = \colim_i F_i$ and $F_i/F_{i-1}$ free.
     Since $\iota_* \Gamma_Z$ is exact, we obtain a chain
     $$
     0 = \iota_* \Gamma_Z F_{-1} \subseteq \iota_* \Gamma_Z F_0 \subseteq \cdots \subseteq \iota_* \Gamma_Z F
     $$
     of dg $C^Z$-modules such that for each $i$, we have $\iota_* \Gamma_Z F_i/\iota_* \Gamma_Z F_{i-1} \cong \iota_* \R\G_Z(F_i/F_{i-1})$, which is semi-flat by what
     was proven above. Since semi-flats are closed under extension, $\iota_* \Gamma_Z F_i$ is semi-flat for each $i$. Finally, we have
     $\iota_* \R\G_Z(F) = \colim_i \iota_* \R\G_Z(F_i)$ and hence is semi-flat. 
     \end{proof}

\begin{proof}[Proof of Theorem \ref{thm711b}]
By Lemma \ref{lem323a} it suffices to prove the canonical map
\begin{equation} \label{E1111a}
\iota_*\G_Z(M) \Lotimes_{C^Z} \iota_* \G_Z(N) \to \G_Z(M \Lotimes_C N)
\end{equation}
in $\cD(A)$ is an isomorphism.
Since $- \Lotimes -$ and $\iota_* \G_Z$ preserve quasi-isomorphisms, and every module is quasi-isomorphic to a semi-free one (Proposition \ref{newprop47}),
we may assume $M$ is semi-free, in which case  $M \Lotimes_C N \cong M \otimes_C N$.
  By Lemma \ref{lem1214} $\iota_* \G_Z(M)$ is semi-flat and hence $\iota_*\G_Z(M) \Lotimes_{C^Z} \iota_* \G_Z(N)  \cong
  \iota_*\G_Z(M) \otimes_{C^Z} \iota_* \G_Z(N)$, by Corollary \ref{hflatcor}.  We may also assume $N$ is semi-free.

  To summarize, it suffices to prove   the canonical map
  \begin{equation} \label{E1214}
  \iota_* \G_Z M \otimes_{C^Z} \iota_* \G_Z N  \to \R\G_Z(M \otimes_C N)
  \end{equation}
  is a quasi-isomorphism whenever $M$ and $N$ are semi-free. The source and target of this natural map, when interpreted as functors in the variable $N$, are exact and preserve all
  colimits. Since $N$ is semi-free, it suffices to prove it is an isomorphism when $N$ is representable: $N = h_X = \Hom_C(X, -)$ for some
  object $X$ of $C$.

By Lemma \ref{lem323e} we have 
$$
\iota_* \G_Z(h_X) \cong \colim( \iota_* \Kos^*(\ug, h_X) \to \iota_* \Kos^*(\ug^2, h_X) \to \cdots)
$$
and thus \eqref{E323var2} gives
$$
\iota_* \G_Z(h_X) \cong \colim( h_{\Kos(\ug, X)} \to h_{\Kos(\ug^2, X)} \to\cdots).
$$
For each $l$ we have natural isomorphisms
$$
\iota_*\G_Z(M) \otimes_{C^Z}  h_{\Kos(\ug^l;X)} \cong \G_Z(M(\Kos(\ug^l; X)) \cong
M(\Kos(\ug^l; X)) \cong \Kos^*(\ug^l; M(X))
$$
with the last isomorphism given by  \eqref{E0101}. This gives
$$
\iota_*\G_Z(M) \otimes_{C^Z} \iota_* \G_Z(h_X) \cong \colim_l(\Kos^*(\ug^l; M(X))) = \G_Z(M(X)).
$$
Since $M \otimes_C h_X \cong M(X)$ (see \eqref{E1029a}), we conclude \eqref{E1214} is a quasi-isomorphism.
\end{proof}

\section{Supports for curved modules and a Theorem of Briggs} \label{SecSupp}

The goal of this section is to prove:

\begin{thm}  \label{thm47}
  Assume $k$ is regular and excellent and $\cA = (A, w)$ is an essentially smooth curved $k$-algebra, and set $W = \Nonreg(\cA) \times_{\Spec(k)} \Nonreg(\cA)$.
  The dg functor $\psi$ defined in \eqref{E1212}  factors as
$$
\Perfdg(\cA)^e \xira{\phi} \Perfdg^W(\cA^e) \subseteq \Perfdg(\cA^e) 
$$
and $\phi$ is a Morta equivalence. 
For any right dg $\Perfdg^W(\cA^e)$-module $M$ and  left dg $\Perfdg^W(\cA^e)$-module  $N$, the canonical map
$$
\phi_* M \Lotimes_{\Perfdg(\cA^e)} \phi_* N \xra{\cong} M \Lotimes_{\Perfdg^W(\cA^e)} N
$$
is an isomorphism in $\cD(A)$. 
\end{thm}

In particular, this theorem gives that the map \eqref{Isomorphism4} is an isomorphism in $\cD(A)$. 
The central ingredient in its proof is the following theorem of Briggs, which generalizes a result of Hopkins and Neeman to the curved setting.
Let us say that a Zariski closed subset $Z$ of $\Spec(A)$ is {\em homogeneous} if $Z = V(I) = \{\fp \in \Spec(A) \mid \fp \supseteq I\}$ for some homogeneous ideal $I$ of $A$.

\begin{thm}[Briggs' Theorem] \label{body_thm_curved_regular}
  Let $\cA = (A, w)$ be a regular curved ring.
There is a bijection
$$
\begin{tikzcd}
  \Big\{\text{thick subcategories of }[\Perfdg(\cA)]\Big\} \ar[r,  shift left=1.5ex, "\sigma"] &
  \Big\{\text{specialization closed, homogeneous subsets of }\Nonreg(\cA)\Big\}, \ar[l,  shift left=1.5ex, "\theta"]
\end{tikzcd}
  $$
  where $\theta(Z)=\{X\in [\Perfdg(\cA)]  \mid \supp_A(X)\subseteq Z\}$ and $\sigma(T)=\bigcup_{X \in T} \supp_A(X)$. 

  In particular, for every Zariski closed, homogeneous subset $Z$  of $\Spec(A)$ there exists an $X \in \Perfdg(\cA)$ such that $Z= \supp_A(X)$,
  and for every object $X \in \Perfdg(\cA)$, its thick closure in $[\Perfdg(\cA)]$ coincides with $[\Perfdg^{\supp_{\cA}(X)}(\cA)]$. 
\end{thm}

See  Appendix \ref{appendix} for the proof of Briggs' Theorem and other results.  (The version of this theorem found in the Appendix,
namely Theorem \ref{thm_curved_regular}, is written in terms of
specialization closed subsets  of the set $\Spec^*(A)$ of homogeneous prime ideals of $A$.
Since for any $X \in \Perfdg(\cA)$ the set $\supp_A(X)$ coincides with the Zariski closure of $\supp_A(Z) \cap
\Spec^*(A)$, the statement given here is equivalent.)

We will also need the following technical result:

\begin{lem} \label{lem818} With the notation as in Theorem \ref{thm47}, for any $X, Y \in \Perfdg(\cA)$, we have
  $$
  \supp_{\cA^e}(\psi(X,Y)) = \supp_A(X) \times_{\Spec(k)} \supp_A(Y).
  $$
\end{lem}

\begin{proof}
Given $X, Y \in \Perfdg(\cA)$, we have
$$
\End_{\cA^e}(\psi(X,Y)) \cong \End_\cA(X) \otimes_k \End_{\cA^{\op}}(Y) = \End_{\cA}(X) \otimes_k \End_{\cA}(Y)^{\op} ,
$$
and $\End_{\cA}(X)^\nat$ and $\End_{\cA}(Y)^\nat$ are finitely generated and projective as graded $A$-modules.
More generally, we prove that if $M$ and $N$ are dg $A$-modules whose underlying graded modules are finitely generated and projective,
then $\supp_{A \otimes_k A}(M \otimes_k N) = \supp_A(M) \times_{\Spec(k)} \supp_A(N)$.

  Given a homogeneous prime $\fp$ of $\Spec(k)$, we set $k_{(\fp)}$ to be the graded algebra obtained by inverting all homogeneous elements of $k \setminus \fp$
  and $\kappa^*(\fp) := k_{(\fp)}/\fp k_{(fp)}$. Note that $\kappa^*(\fp)$ is  a graded field, meaning that every non-zero homogeneous element is a unit.

  We claim that  for any homogeneous prime $\fp$ and dg $k$-module $L$ such that $L^\nat$ is finitely generated and projective,
  $\fp$ belongs to $\supp_k(L)$ if and only if $H(L \otimes_k \kappa^*{(\fp)}) \ne 0$.  To prove this, it suffices to  assume $k = k_{(\fp)}$ (i.e, $k$ is graded local) and show
  $H(L) \ne 0$ if and only if $H(L \otimes_k k/\fp) \ne 0$. 
Since $k$ is regular,
  $\fp$ is generated by a regular sequence $x_1, \dots, x_d$
  of homogeneous elements. Set $L_i = L/(x_1, \dots, x_i)L$. Since $L^\nat$ is projective, for each $i$ we have a short exact sequence
  $$
  0 \to L_i \xra{x_{i+1}} L_i \to L_{i+1} \to 0
  $$
  of dg $k$-modules. The long exact sequence in homology shows that $H(L_i) \ne 0$ if and only if $H(L_{i+1}) \ne 0$, and the claim follows.

  Since $A$ is essentially smooth over $k$, both $A$ and $A\otimes_k A$ are also regular. So, the claim just proven
  applies to modules over these graded rings as well.

Since $M$ and $N$ are finitely generated,
both $\supp_{A \otimes_k A} (M \otimes_k N)$ and  $\supp_A(M) \times_{\Spec(k)} \supp_A(N)$
are homogeneous Zariski closed subsets of $\Spec(A \otimes_k A) = \Spec(A) \times_{\Spec(k)} \Spec(A)$.
  In particular, it suffices to prove that a homogeneous prime $\fq$ ideal of $A \otimes_k A$ belongs to 
$\supp_{A \otimes_k A} (M \otimes_k N)$ if and only if it belongs to $\supp_A(M) \times_{\Spec(k)} \supp_A(N)$.
For such a $\fq$, let $\fp$ be its image in $\Spec^*(k)$, and set $A' = A \otimes_k k'$,
$M' = A' \otimes_A M$ and $N' = A' \otimes_A N$. By the claim proven above, 
$\fq$ lies in $\supp_{A \otimes_k A}(M \otimes_kN)$ if and only if it lies in $\supp_{A' \otimes_{\kappa^*(\fp)} A'}(M' \otimes_{\kappa^*(\fp)}N')$,
and similarly for $\supp_A(M) \times_{\Spec(k)} \supp_A(N)$. In other words, without loss of generality, we may assume $k$ is a graded field.

When $k$ is a graded field, we have the K\"unneth isomorphism $H^*(M \otimes_k N) \cong H^*(M) \otimes_k H^*(N)$, and the result follows
from the fact that
$$
\supp_{A \otimes_k A}(H^*(M) \otimes_k H^*(N)) = \supp_A H^*(M) \times_{\Spec(k)} \supp_A H^*(N),
$$
which holds since $H^*(M)$ and $H^*(N)$ are finitely generated, graded $A$-modules.
\end{proof}

\begin{proof}[Proof of Theorem \ref{thm47}]
Lemma \ref{lem818} gives
\begin{equation} \label{E817}
\supp_{\cA^e}(\psi(X,Y)) = \supp_A(X) \times_{\Spec(k)} \supp_A(Y).
\end{equation}
Since every object of $\Perfdg(\cA)$ is supported on $\Nonreg(A/w)$ (by Briggs' Theorem),
we have $\supp_{\cA^e}(\psi(X,Y)) \subseteq W$, and this proves that $\psi$ factors as claimed. 

By Proposition \ref{thickprop}, to show $\phi$ is a Morita equivalence, it suffices
to show the thick closure of the image of $\psi$ in  $[\Perfdg^W(\cA^e)]$ is all of  $[\Perfdg^W(\cA^e)]$.  
Since $k$ is excellent and $A$ is essentially of finite type over $k$, $A$ is also excellent. It follows that $\Nonreg(A/w)$ is a homogeneous closed subset of $\Spec(A)$.
In particular, 
Briggs' Theorem gives that there exist an object $X \in \Perfdg(\cA)$ with $\supp(X) = \Nonreg(A/w)$; see Corollary \ref{corA12}. 
Moreover, using Lemma \ref{lem818} again, we have  $\supp(\psi(X,X)) = W$, and thus another application of Briggs' Theorem gives 
$$
\Thick_{[\Perfdg(\cA^e)]}(\psi(X,X)) =  [\Perfdg^W(\cA^e)].
$$

We have shown $\phi$ is a Morita equivalence, and
the final assertion is thus a consequence of Proposition \ref{Moritaprop}. 
\end{proof}

\section{Completion of the proof of the first part of Theorem \ref{MainTheorem}} \label{sec:nochar}

In this section, we prove the portion of Theorem \ref{MainTheorem} that is valid without any assumption on the characteristic of $k$:

 \begin{thm} \label{thm1227} If $k$ is regular and excellent and $\cA$ is a smooth curved $k$-algebra,  then there is an isomorphism
$$
HH(\Perfdg(\cA)) \cong \R\Gamma_{\Nonreg(\cA)}(\Omega^\cdot_{A/k}, dw),
$$
in the derived category of dg $A$-modules. 
\end{thm}

Before starting the proof, we set up some notation and establish a preliminary result.

Let $\mu: A^e \to A$ be the multiplication map, set $I = \ker(\mu)$, and recall that $I/I^2 \cong \Omega^1_{A/k}$. 
Let $P$ be any finitely generated projective $A^e$-module that is a lift of $\Omega^1_{A/k}$ along $\mu$, in the sense that there is an isomorphism $P/IP \cong \Omega^1_{A,k}$.
Such a $P$ exists since $\mu$ is split surjective; e.g., we may take $P = A \otimes_k \Omega^1_{A/k}$.
Then there is an induced surjection $P \onto I/I^2$, and, since $P$ is projective, this map lifts along $I \onto I/I^2$ 
to a map $q: P \to I$ of $A^e$-modules. Regarding $q$ as a map with target $A^e$, set $K = \Kos_{A^e}(q) = (\bigwedge_{A^e}(P), \del_q)$, the Koszul complex associated to $q$. 

Since $A$ is essentially smooth over $k$, the ideal $I$ is locally generated by a regular sequence. 
Hence, if the map $q$ were onto, then $K$ would be a free resolution of $A^e/I \cong A$. It need not be the case that $q$ is onto, but it is true locally on points in $V(I)$.
In detail, 
setting $C = I/\im(q)$ and $J = \ann_{A^e}(C)$,  since $\overline{q}: P/IP \cong I/I^2$ is an isomorphism, we have $C/IC = 0$ and thus  Nakayama's Lemma gives
that $V(I) \cap V(J) = \emptyset$. So, $J + I = A^e$ and, since $I$ and $J$ are homogeneous ideals, we have $1 = \a + \b$ for some elements $\a \in J$ and
$\b \in I$ that are homogeneous of degree $0$. It follows that $q: P[1/\a] \onto  I[1/\a]$ is a surjection of $A^e[1/\a]$-modules,  and that $K[1/\a] = \Kos_{A^e[1/\a]}(q)$ is a resolution of
$(A^e/I)[1/\a]  \cong A$. (Note that $\a$ is a unit modulo $I$.)

Since $P[1/\a]$ surjects onto $I[1/\a]$ and $w^e \in I$, there is a $\g \in P[1/\a]$ of degree $2$ satisfying $q(\g) = w^e$. We set $\tilde{K}$ to be the curved $\cA^e$-module
$$
\tilde{K} = (\extpower_{A^e[1/\a]} P[1/\a], \del_q + \l_\g)
$$
where $\del_q$ is the differential for the Koszul complex and $\l_\g$ is left multiplication by $\g$.
$\tilde{K}$ is indeed a curved module over $\cA^e$ because $d_q^2 = 0$, $\l_\g^2 = 0$, and $[d_q, \l_\g] = \l_{q(\g)} = \l_{w^e}$.
Moreover, there is a canonical morphism of curved $\cA^e$-modules $p: \tilde{K} \to A$. 

\begin{lem} \label{lem43} The map $h^{\tilde{K}} \to h^A$ of dg $\Perfdg(\cA^e)$-modules is a quasi-isomorphism.  
 \end{lem}

 \begin{proof} For any $X \in \Perfdg(\cA^e)$, we need to show that cone of the map of complexes $\Hom_{\cA^e}(X, \tilde{K}) \to \Hom_{\cA^e}(X, A)$ is exact. 
As a chain complex of abelian groups, this cone may be identified with the totalization of the diagram
   $$
     \begin{tikzcd}
        & \vdots \ar[d] & \vdots \ar[d] \ar[ddl] & & \vdots \ar[d] \ar[ddl] & \vdots \ar[d]  \\
        0 \ar[r] & B_n^{m-1} \ar[r] \ar[d] &  B_{n-1}^{m-1}  \ar[r] \ar[d] \ar[ldd]
        & \cdots \ar[r] & B_{0}^{m-1} \ar[r]\ar[d] \ar[ldd] & B_{-1}^{m-1} \ar[r]\ar[d]  & 0 \\
        0 \ar[r] & B_n^{m} \ar[r] \ar[d] &  B_{n-1}^{m}  \ar[r] \ar[d] \ar[ldd]
        & \cdots \ar[r] & B_{0}^{m} \ar[r]\ar[d] \ar[ldd] & B_{-1}^{m} \ar[r]\ar[d]  & 0 \\
        0 \ar[r] & B_n^{m+1} \ar[r] \ar[d] &  B_{n-1}^{m+1}  \ar[r] \ar[d] \ar[ldd]
        & \cdots \ar[r] & B_{0}^{m+1} \ar[r]\ar[d] \ar[ldd] & B_{-1}^{m+1} \ar[r]\ar[d]  & 0 \\
        0 \ar[r] & B_n^{m+2} \ar[r] \ar[d] &  B_{n-1}^{m+2}  \ar[r] \ar[d] 
        & \cdots \ar[r] & B_{0}^{m+2} \ar[r]\ar[d]  & B_{-1}^{m+2} \ar[r]\ar[d]  & 0 \\
     & \vdots & \vdots & \vdots & \vdots & \vdots & \vdots \\
     \end{tikzcd}
     $$
    where $B_j^i := \Hom{\cA^e}(X, \tilde{K}_j)^i$ for $0 \leq j \leq n$,
     $B_{-1}^i := \Hom{\cA^e}(X, A)^i$, the horizontal arrows are induced by $\del_q$ and the map $p$, the vertical arrows are induced by the differential on $X$, and the slanted
     arrows are induced by $\lambda_\gamma$. 
     Since $X \in \Perfdg(\cA)$, each row in this diagram is exact, and thus its totalization is also exact.
     \end{proof}

\begin{lem} \label{lem416} For a regular curved ring $\cA = (A, w)$, suppose $F$ is a left curved $\cA$-module such that $F^\nat$ is flat as an $A$-module and $N$ be a right curved $\cA$-module
  such that $N^\nat$ is finitely generated as an $A$-module. Then there is an isomorphism 
  $$
  h^F \Lotimes_{\Perfdg(\cA)} t^N \cong N \otimes_\cA F
  $$
  in the derived category of dg $A$-modules. (Recall $t^N$ is the left dg $\Perfdg(\cA)$-module given by $t^N(X) = N \otimes_{\cA} X$.)
\end{lem}

\begin{proof} Since $N^\nat$ is finitely generated over $A$ and $A$ is regular, $\pd_A(N^\nat) = m < \infty$, and it follows that we can find an exact sequence of right curved $\cA$-modules of the form
  $$
  0 \to X_m \to \cdots \to X_0 \to N \to 0
  $$
  such that $X_i$ is perfect for each $i$. Since $\Hom_{\cA}(Y, -)$ is exact whenever $Y$ is perfect, the sequence
  \begin{equation} \label{E416}
  0 \to t^{X_m} \to \cdots \to t^{X_0} \to t^{N} \to 0
  \end{equation} 
 of dg $\Perfdg(\cA)$-modules is exact. For any perfect module $X$ we have $t^X \cong h_{X^*}$ where $X^* = \Hom_{\cA}(X, A)$.
Set  $P := \Tot(0 \to h_{X_m^*} \to \cdots \to h_{X_0^*} \to 0)$. Then $P$ is semi-free and \eqref{E416} gives a quasi-isomorphism $P \xra{\sim} t^N$ of dg $\Perfdg(\cA)$-modules. Thus
we have an isomorphism  $$
  h^F \Lotimes_{\Perfdg(\cA)} t^N \cong   h^F \otimes_{\Perfdg(\cA)} P
  $$
  in the derived category.
The isomorphisms $h^F \otimes_{\Perfdg(\cA)} h_{X_i^*} \cong h^F(X_i^*) \cong \Hom_{\cA}(X_i^*, F) \cong X_i \otimes_{\cA} F$ 
give an isomorphism
  $$
  h^F \otimes_{\Perfdg(\cA)} P
  \cong \Tot(0 \to X_m \otimes_{\cA} F \to \cdots \to X_0 \otimes_{\cA} F \to 0).
  $$
  Since $F^\nat$ is $A$-flat, the sequence
  $$
0 \to X_m \otimes_{\cA} F \to \cdots \to X_0 \otimes_{\cA} F \to N \otimes_{\cA} F \to 0
$$
is exact and hence we have a quasi-isomorphism
$$
\Tot(0 \to X_m \otimes_{\cA} F \to \cdots \to X_0 \otimes_{\cA} F \to 0) \xra{\sim} N \otimes_{\cA} F.
$$
\end{proof}

Let us now prove the theorem:

\begin{proof}[Proof of Theorem \ref{thm1227}]  Theorems \ref{thm711b} and \ref{thm47} give the isomorphism 
  $$
  HH(\Perfdg(\cA)) \cong \R\G_W \left( h^A \Lotimes_{\Perfdg(\cA^e)}  t^A\right).
  $$ 
  (The assumption that $k$ is excellent implies that
  $W$ is closed subset of $\Spec(A)$, and hence Theorem \ref{thm711b} applies.)
  Lemmas \ref{lem43} and \ref{lem416} give the isomorphisms 
  $$
  h^A \Lotimes_{\Perfdg(\cA^e)} t^A \cong h^{\tilde{K}} \Lotimes_{\Perfdg(\cA^e)} t^A \cong A \otimes_{\cA^e} \tilde{K}.
  $$
The graded $A^e$-module underlying $A \otimes_{\cA^e} \tilde{K}$ is
$$
A \otimes_{A^e} \extpower_{A^e[1/\alpha]} P[1/\alpha] \cong \extpower_A(P/IP) \cong \Omega^\cdot_A.
$$
The part of the differential on $\tilde{K}$ given by $\del_q$ vanishes on $A \otimes_{\cA^e} \tilde{K}$,
and the part given by $\l_\g$ corresponds to multiplication by $dw \in \Omega^1_A$.
Thus we have an isomorphism
$$
h^A \Lotimes_{\Perfdg(\cA^e)} t^A \cong (\Omega^\cdot_{A/k}, dw)
$$
in $\cD(A)$. 
\end{proof}

\section{Naturality and Connes operators in characteristic zero} \label{sec:char0}

In this section we complete the proof of Theorem \ref{MainTheorem} by establishing the
assertions regarding naturality and compatibility with the Connes' operators
in the case  when $k$ is a $\Q$-algebra. A key ingredient in the proof is the notions of type II derived tensor products and type II Hochschild homology, developed by Positselski \cite{Positselski} and Polishchuk-Positselski \cite{PP}.
In detail, we will reinterpret the isomorphism $HH(\Perfdg(\cA)) \cong \R\Gamma_{\Nonreg(\cA)} \left(h^A \Lotimes_{\Perfdg(\cA^e)} t^A\right)$ in terms of type II Hochschild homology.

\subsection{Type II tensor products}
We start be recalling some definitions and basic properties.  References for this material are \cite{Positselski} and \cite{PP}. 

A {\em cdg $k$-algebra} consists of a triple $(A,d, w)$ where $A$ is a graded $k$-module, $d$ a degree one $k$-linear map, and $w$ is a degree two element such that $d^2(a) = [w,
a] = wa - aw$ for all $a \in A$ and $d(w) = 0$. Note that a curved algebra is the special case of this more general concept with $d = 0$. 
More generally, a {\em $k$-linear cdg (curved differential graded) category}  $\cC$ consists a collection of objects $\ob \cC$,
for each pair of objects $X$ and $Y$, a pair $\Hom_\cC(X, Y) = (\Hom_\cC(X,Y)^\nat, d_{X,Y})$ consisting of a graded $k$-module and a degree one $k$-linear endomorphism (called the
{\em pre-differential}),
for each single object $X$ a degree two
element $w_X$ of $\Hom_\cC(X, Y)$ (called the {\em curvature}), and for each triple of objects $X, Y, Z$ a composition rule 
$\Hom_\cC(X, Y)^\nat \otimes \Hom_\cC(Y, Z)^\nat \to \Hom_\cC(X, Z)^\nat$. The compositions rules are required to be associative and unital just as for dg categories, 
but the  maps $d_{X, Y}$ need not square to $0$.
Instead, the following relations hold: (1) $d^2 = [w, -]$; that is,
for all pairs of objects $X, Y$ and elements $f \in \Hom_{\cC}(X, Y)$, we have $d_{X, Y}^2(f) = w_Y \circ f - f \circ w_X$ and (2) $d(w) = 0$; that is, for
each object $X$, we have $d_{X, X}(w_X) = 0$. 
The endomorphisms of any object of a cdg category form a cdg algebra, and a cdg category on just one object may be identified with a cdg algebra.

For example, the collection of {\em precomplexes} over $k$, written $\Pre(k)$, is a cdg category. Its objects are pairs $(M, d_M)$ consisting of a graded $k$-module $M$ and a
$k$-linear endomorphism $d_M$ of degree one, with no condition on $d_M^2$. The homs and  composition rules are given by the same formulas as for the dg category $\Moddg(k)$, 
and we define the curvature of $(M, d_M)$ to be $w_M := d_M^2 \in \End_k(M)$.

A (strict) cdg functor between two cdg categories $\cC$ and $\cD$ consists of a function $F: \ob \cC \to \ob \cD$ on objects and for each pair of objects $X, Y$ of $\cC$
a morphism of graded $k$-modules $F_{X, Y}: \Hom_\cC(X,Y)^\nat \to \Hom_\cD(F(X), F(Y))^\nat$ that preserves compositions, identity elements, and curvatures, and
commutes with the  pre-differentials.
There is a more general notion of a ``non-strict'' functor joining cdg categories, but since it only makes a brief appearance in this paper in a special setting,
we will not define it carefully. 

A {\em left cdg module} on a $k$-linear cdg category $\cC$ is a (strict) cdg functor from $\cC$ to $\Pre(k)$.
This generalizes the notion of a left curved module on a curved algebra used throughout this paper.
The collection of left cdg modules over $\cC$ forms a (non-curved) dg category, generalizing the fact that the collection of curved modules over a curved algebra forms a
dg category.

There is also the  notion of a {\em quasi-module} on an arbitrary  cdg category, defined by using the notion of a non-strict cdg functor
taking values  in $\Pre(k)$. Since we will only need this concept for curved algebras, we only define it at that level of generality:
Given a curved $k$-algebra $\cA = (A, w)$, 
a {\em left quasi-module} over $\cA$ consists of a left graded $A$-module $M$ and an $A$-linear map $d_M$ on $M$ of degree
one, but there is no condition on $d_M^2$.
Such a quasi-module is {\em perfect} if $M$ is finitely generated and projective as a graded  $A$-module.
The collection of perfect quasi-modules form an $A$-linear  cdg category, written $\qPerfcdg(\cA)$. The homs and pre-differentials for a pair of objects are defined just as for
$\Perfdg(\cA)$, and  
for each $P \in \qPerfcdg(\cA)$, the required curvature element  is defined to be $w_P := d_P^2 - \l_w \in \End(P)$, where $\l_w$ denotes left multiplication by $w$. 
Observe that $\Perfdg(\cA)$ is the full subcategory of $\qPerfcdg(\cA)$ consisting of the objects with trivial curvature. 

Although the $k$-algebra $A$ equipped with the trivial differential is {\em not} an object of the dg category $\Perfdg(\cA)$,
it {\em does} determine an object of $\qPerfcdg(\cA)$.
Moreover, its endomorphisms form the curved algebra $\cA^\op = (A, -w)$.
Abusing notation a bit, we identity $\cA^\op$ with the full cdg subcategory of $\qPerfcdg(\cA)$ on this object.
Similarly, we regard $\cA$ as a full cdg subcategory of $\qPerfcdg(\cA^\op)$
and $(\cA^e)^\op$ as a full cdg subcategory of $\qPerfcdg(\cA^e)$.

Each $k$-linear cdg category $\cC$ admits an opposite category $\cC^\op$ and enveloping category $\cC^e = \cC \otimes_k \cC^\op$,
defined just as for dg categories, but with curvature taken into account: For $X \in \ob(\cC)$, its curvature when regarded as an object of $\cC^\op$
is $-w_X$, and the curvature of a pair $(X, Y) \in \cC^e$ is $w_X \otimes 1 - 1 \otimes w_Y$. The curved algebras $\cA^\op$ and $\cA^e$ introduced before are special cases of
this. 

The fully-faithful dg functor $\psi: \Perfdg(\cA)^e \into \Perfdg(\cA^e)$ defined in \eqref{E1212} extends to a (strict) fully-faithful cdg functor
$\tpsi: \qPerfcdg(\cA)^e \into \qPerfcdg(\cA^e)$ using the same formulas on objects and homs. This is indeed a cdg functor since, for each pair $(X, Y)$,
it sends $w_X \otimes 1 - 1 \otimes w_Y = (d_X^2 - \l_w) \otimes 1 - 1 \otimes (d_Y^2 - \l_w)$ to $(d_X \otimes 1 - 1 \otimes d_Y)^2 - \l_{w \otimes 1 - 1 \otimes w}$. 
Observe that $\tpsi$ sends $(A^\op, A) \in \qPerfcdg(\cA)^e$ to  $(\cA^e)^{\op} \in \qPerfcdg(\cA^e)$, and the map on endomorphism cdgas is an isomorphism.
We thus have a diagram of fully-faithful functors of cdg categories
\begin{equation} \label{E414}
  \begin{tikzcd}
\Perfdg(\cA)^e \ar[r, hook, "\psi"] \ar[d, hook] & \Perfdg(\cA^e) \ar[d, hook] \\
\qPerfcdg(\cA)^e \ar[r, hook,  "\tpsi"] & \qPerfcdg(\cA^e) \\
(\cA^\op, \cA) \ar[r, "\cong"] \ar[u, hook] & (\cA^e)^\op \ar[u, hook] \\
\end{tikzcd}
\end{equation}
in which all vertical functors are inclusions.

For any cdg category $\cC$, left cdg $\cC$-module $M$, and a right cdg $\cC$-module $N$,
their {\em type II derived tensor product}, written $M \LotimesII_{\cC} N$,
is defined (somewhat loosely speaking) by taking direct product totalization of the bi-complexes obtained from a choice of
resolution; see \cite[p. 5325]{PP} for the precise definition, where $\LotimesII$ is written as $\Tor^{II}$.
(There is also an ordinary derived tensor product of $M$ and $N$, but it often vanishes.) 
When $\cC$ is a dg category (i.e., when all the curvature elements are trivial), there is a natural map
$$
M \Lotimes_{\cC} N \to M \LotimesII_{\cC} N
$$
joining the two kinds of derived tensor product. 
This map is a weak equivalence under special circumstances, but in general
the two theories enjoy rather different formal properties. For instance,
the type II derived tensor product does {\em not}, in general, preserve quasi-isomorphisms.
On the other hand, it is well-behaved with respect to ``pseudo-equivalences'' of cdg categories (see \cite[p. 5326]{PP}),
a property that will be exploited in the proof of Proposition \ref{prop416} below.

\subsection{Type II interpretations}
Recall that we have dg $\Perfdg(\cA^e)$-modules $h^A = \Hom_{\Perfdg(\cA^e)}(-, A)$ and $t^A = - \otimes_{\cA^e} A$.
They extend to cdg modules on $\qPerfcdg(\cA^e)$ using the same formulas. Abusing notation a bit, we use the same notation for these cdg modules.

\begin{prop} \label{prop416} With $k$ and $\cA$ as in Theorem \ref{MainTheorem}, there is a commutative diagram in $\cD(A)$ of the form
   \begin{equation} \label{E415} 
   \begin{tikzcd}
     \Deltaright\Lotimes_{\Perfdg(\cA)^e} \Deltaleft \ar[r, "\cong"] \ar[d, "\can"] & \R\Gamma_{\Nonreg(\cA)} \left(h^A \Lotimes_{\Perfdg(\cA^e)} t^A\right) \ar[r, "\can"]  &
     h^A \Lotimes_{\Perfdg(\cA^e)} t^A \ar[d, "\can", "\cong"'] \\
  \Deltaright\LotimesII_{\Perfdg(\cA)^e} \Deltaleft \ar[rr, "\cong"] \ar[d, "\cong"] &&  h^A \LotimesII_{\Perfdg(\cA^e)} t^A \ar[d, "\cong"] \\
  \Deltaright \LotimesII_{\qPerfcdg(\cA)^e} \Deltaleft \ar[rr, "\cong"]  &&  h^A \LotimesII_{\qPerfcdg(\cA^e)} t^A  \\
A \LotimesII_{\cA^e} A  \ar[rr, "="] \ar[u, "\cong"] &&   A \LotimesII_{\cA^e} A \ar[u, "\cong"] \\
 \end{tikzcd}
 \end{equation}
with isomorphisms as indicated. 
\end{prop}

\begin{proof}
The restrictions of $h^A$ and $t^A$ along $\tpsi$ give the cdg modules $\Deltaright$ and $\Deltaleft$, just as with $\psi$. 
Moreover,  the restrictions of each of  $h^A$ and $t^A$ to the one object category $(\cA^e)^\op$ give 
$A$ (regarded as a left or right curved $\cA^e$-module). Combining these facts with  diagram \eqref{E414}
and the naturality of the map from $\Lotimes$ to $\LotimesII$ gives the commutativity of \eqref{E415}  and the fact that the bottom row may be taken to be the identity map. 

The upper-left map in \eqref{E415} is an isomorphism by Theorems \ref{thm711b} and \ref{thm47}.

As in the proof of Lemma \ref{lem416}, since $A^e$ is regular and $A$ is finitely generated, we can construct
an exact sequence $0 \to P_m \to \cdots \to P_0 \to A \to 0$ with $P_i \in \Perfdg(\cA^e)$ for each $i$, and this leads to an exact sequence
\begin{equation} \label{E423}
0 \to h^{P_m} \to h^{P_{m-1}} \to \cdots \to h^{P_1} \to h^{P_0} \to h^A \to 0 
\end{equation}
of dg $\Perfdg(\cA^e)$-modules. That is, $h^A$ admits a bounded resolution by representable modules.
It is clear from the definitions of $\Lotimes$ and $\LotimesII$ that they agree if one of the arguments is representable.
Moreover,  both $\Lotimes$ and $\LotimesII$ send short exact sequence to distinguished triangles (see \cite[p. 5326]{PP}).
It thus follows from \eqref{E423} the canonical map
$h^A \Lotimes_{\Perfdg(\cA^e)} t^A \to h^A \LotimesII_{\Perfdg(\cA^e)} t^A$ is an isomorphism.

The remaining vertical maps are isomorphisms since each vertical inclusion in \eqref{E414} is a pseudo-equivalence and,  as proven in \cite[p. 5326]{PP},
pseudo-equivalences induce isomorphisms on type II tensor products.

It follows that the two middle horizontal maps are also isomorphisms.
 \end{proof}

 For a $k$-linear cdg category $\cC$ we define its {\em type two Hochschild homology} to be 
 $$
 HH^{II}(\cC) := \Deltaright \LotimesII_{\cC^e} \Deltaleft \in \cD(k).
 $$
There is a natural morphism $HH(\cC) \to HH^{II}(\cC)$ in $\cD(k)$ induced by the map from $\Lotimes$ to $\LotimesII$. With this notation,   Proposition \ref{prop416}
immediately implies:

 \begin{cor}  With $k$ and $\cA$ as in Theorem \ref{MainTheorem}, the canonical map $HH(\Perfdg(\cA)) \to HH^{II}(\Perfdg(\cA))$ induces an isomorphism
   $$
   HH(\Perfdg(\cA)) \xra{\cong} \R\Gamma_{\Nonreg(\cA)} HH^{II}(\Perfdg(\cA))
   $$
   in the derived category of dg $A$-modules.
 \end{cor}

 \subsection{Hochschild complexes}
 In order to establish the naturality of our isomorphism and to even define the Connes' operators, we need a concrete model for the Hochschild homology complex.

 For a small cdg category $\cC$, the {\em Hochschild complex}, written $\Hoch(\cC)$, is the explicit dg $k$-module defined as a follows. The underlying graded $k$-module
 is the direct {\em sum} totalization of the collection $\{\Sigma^n \cH_n(\cC) \}_{n \geq 0}$ where
  $$
 \cH_n(\cC) := \bigoplus_{X_0, \dots, X_n} \Hom_C(X_1, X_0) \otimes_k \Hom_C(X_2, X_1) \otimes_k  \cdots \Hom_C(X_{n}, X_{n-1}) \otimes_k \Hom_C(X_0, X_n) 
 $$
 The differential on $\Hoch(\cC)$ has three components, arising from the classical Hochschild differenital $b$, maps coming from the
  (pre)differentials on $\Hom_{\cC}(X,Y)$, and maps induced by the curvature elements.
 
 For instance, suppose $\cC$ has just one object, and thus may be identified with a cdga $(A, d, h)$. 
If we further specialize to the case when $d = 0 = h$ and $A$ is concentrated in degree $0$, 
 then $\Hoch(\cC) = \Hoch(A)$ is the classical Hochschild complex
 $$
 \cdots \xra{b} A^{\otimes n+1}  \xra{b}  A^{\otimes n}  \xra{b}  \cdots \xra{b} A^{\otimes 2}  \xra{b}  A \to 0
 $$
 where
 $$
 b(a_0 \otimes a_1 \otimes \cdots \otimes a_n)
= (-1)^n a_n a_0 \otimes a_1 \otimes \cdots \otimes a_{n-1} + 
 \sum_{i=0}^{n-1} (-1)^i a_0 \otimes a_1 \otimes \cdots \otimes a_i a_{i+1} \otimes \cdots \otimes a_n.
 $$
If $A$ is graded, the signs are suitably modified, we get a complex of graded $A$-modules,  and $\Hoch(A)$ is its direct sum totalization.
More generality, if $A$ is a dga,  the  formula for the differential also includes terms of the form
$$
\pm a_0 \otimes a_1 \otimes \cdots \otimes d(a_i)  \otimes \cdots \otimes a_n
$$
 and if $A$ is a cdga, the differential includes terms of the form 
 $$
 \pm a_0 \otimes a_1 \otimes \cdots \otimes a_i \otimes h_A \otimes a_{i+1}   \otimes \cdots \otimes a_n.
 $$
Finally, in full generality, the differential on $\Hoch(\cC)$ is given by these same formulas, but now interpreting the $a_i$'s to be elements of
 $\Hom_{\cC}(X_{i+1}, X_{i})$ (with $X_{n+1} := X_0$).  
 See \cite[\S 2.4]{PP} or \cite[\S 3.1]{BWChern} for the precise formulas. 

$\Hoch(-)$ is covariantly functorial for (strict) cdg functors, taking values in the category of dg $k$-modules (not merely the derived
category of such); see \cite[\S 2.4]{PP}. 

We also have the {\em type two Hochschild complex}, written $\HochII(\cC)$,  defined in the same way, except that one takes the {\em direct product} totalization of the family of
graded $k$-modules $\{\cH_n(\cC)\}_{n \geq 0}$. 
$\HochII(-)$ is also covariantly functorial for cdg categories, taking values in the category of dg $k$-modules, and there is a natural map $\Hoch(-) \to \HochII(-)$.

\begin{prop} \label{prop415c} \cite[Proposition A in \S 2.4 and  p.5333]{PP} 
  For each small $k$-linear cdg category $\cC$, there are isomorphisms
   $$
   \Hoch(\cC) \cong HH(\cC) = \Deltaright \Lotimes_{\cC} \Deltaleft
   \and
      \HochII(\cC) \cong HH^{II}(\cC) = \Deltaright \LotimesII_{\cC} \Deltaleft
   $$
   in the derived category of dg $k$-modules.
   \end{prop}

   \begin{defn} When $k$ is a $\Q$-algebra, the {\em Chern character map}
     $$
     ch^{II}:  \HochII(\cA) \to (\Omega^\cdot_A, {dw})
     $$
is the map of dg $A$-modules determined by 
     $$
     ch^{II}(a_0 \otimes a_1 \otimes \cdots \otimes a_n) = \frac{1}{n!} a_0 da_1 \cdots da_n \in \Omega^n_{A/k}.
     $$
     We write $ch$ for the restriction of this map along the canonical map  $\Hoch(\cA) \to \HochII(\cA)$.
   \end{defn}
   
The following was proven by Efimov:
   
   \begin{prop} \cite[3.13]{Ef} \label{prop415b}
     Assume $k$ is a $\Q$-algebra and $\cA = (A, w)$ is a smooth curved $k$-algebra. Then the Chern character map
     $ch^{II}: \HochII(\cA) \to (\Omega^\cdot_A, {dw})$
     is a quasi-isomorphism of dg $A$-modules.
\end{prop}
   
\subsection{Naturality and the Connes operator}
For any $k$-linear cdg category $\cC$, the dg $k$-modules $\Hoch(\cC)$ and $\HochII(\cC)$ each admit a  {\em Connes operators} $B$. These are $k$-linear endomorphisms of
homological degree $1$, determined by the formula 
$$
B|_{\cH_n} = (1-\tau^{-1}_{n+2}) \circ s_0 \circ \sum_{l = 0}^n \tau^l_{n+1} 
$$
where $s_0: \cH_n \to \cH_{n+1}$ is the ``extra degeneracy map'' and $\tau_m: \cH_m \to \cH_m$ is (up to a sign) given by cyclic permutation;
see \cite[\S 3.1]{BWChern} for more details.
These operators are natural for (strict) functors of $k$-linear cdg categories.

It is important to be aware that the Connes operators need not preserve additional linearity that $\Hoch(\cC)$ may possess. For instance, $\Hoch(\Perfdg(\cA))$ is a dg $A$-module,
but the Connes operator $B$ is {\em not} $A$-linear, merely $k$-linear.

When $k$ is a $\Q$-algebra,  for any smooth curved $k$-algebra $\cA = (A, w)$, we have established a sequence 
\begin{equation} \label{E417} 
\Hoch(\Perfdg(\cA)) \xra{\can} \HochII(\Perfdg(\cA)) \xra{\sim} \HochII(\qPerfcdg(\cA)) \xla{\sim}  \HochII(\cA) \xra{\sim} (\Omega^\cdot_A, dw)
\end{equation}
of maps of dg $A$-modules, with the arrows labeled by $\sim$ being quasi-isomorphisms. (Only the last map requires $k$ to be a $\Q$-algebra.)
Moreover, each map in this sequence is natural for (strict)  morphisms of curved algebras, the first three maps are compatible with the Connes operators, and
the final map, $\HochII(\cA) \xra{\sim} (\Omega^\cdot_A, dw)$, relates the Connes operator with the de Rham differential \cite[3.14]{Ef}. 

Since the first map in \eqref{E417} induces a quasi-isomorphism
$\Hoch(\Perfdg(\cA)) \xra{\sim} \R\Gamma_{\Nonreg(\cA)}\HochII(\Perfdg(\cA))$, we obtain a ``zig-zag'' diagram of natural quasi-isomorphisms of dg $A$-modules
joining $\Hoch(\Perfdg(\cA)$ and $\R\Gamma_{\Nonreg(\cA)}(\Omega^\cdot_A, dw)$.
This proves the naturality assertion in Theorem \ref{MainTheorem}, 
but in order to deal with the non-$A$-linear Connes operators, we will need to be more careful about the role of local cohomology.

For any tuple  $\ug = (g_1, \dots, g_m)$, setting $W = V(g_1, \dots, g_m)$,   recall that 
$$
\R\G_W(\Omega^\cdot_{A/k}, dw) = \Tot\left((\Omega^\cdot_{A/k}, dw) \to \bigoplus_i (\Omega^\cdot_{A/k}, dw)[1/g_i] \to \cdots\right)
$$
Using the natural isomorphisms $\Omega^\cdot_{A/k}[1/g]  \cong \Omega^\cdot_{A[1/g]/k}$, we may identify this with the totalization of
$$
(\Omega^\cdot_{A/k}, dw) \to \bigoplus_i (\Omega^\cdot_{A[1/g_i]/k}, dw)\to \cdots.
$$
The maps in this sequence are compatible with the various de Rham differentials, and thus
$\R\Gamma_W(\Omega^\cdot_{A/k}, dw)$ acquires a de Rham differential. 

In a similar way, for each homogeneous element $g$ of $A$, we have a  canonical map
$$
\Hoch(\Perfdg(\cA))[1/g] \to \Hoch(\Perfdg(\cA[1/g])),
$$
which we claim is a quasi-isomorphism. This follows from the isomorphisms
$$
\Hoch(\Perfdg(\cA))[1/g] \cong \R\Gamma_{\Nonreg(\cA)}(\Omega^\cdot_{A/k}, dw)[1/g]
\and
\Hoch(\Perfdg(\cA[1/g])) \cong R\Gamma_{\Nonreg(\cA[1/g])}(\Omega^\cdot_{A[1/g]/k}, dw) 
$$
in the derived category of dg $A[1/g]$-modules. 
We are thus justified in setting
$$
\begin{aligned}
  \R\G_W \Hoch(\Perfdg(\cA)) = & \Tot(0 \to \Hoch(\Perfdg(\cA)) \to \bigoplus_i \Hoch(\Perfdg(\cA[1/g_i])) \to \cdots \\
  & \to \Hoch(\Perfdg(\cA[1/g_1\cdots g_n])) \to 0). \\
\end{aligned}
$$

Taking $W = \Nonreg{\cA}$ and using these interpretations of local cohomology, we see that upon applying $\R \Gamma_{\Nonreg(\cA)}$ to \eqref{E417} we
obtain a sequence of quasi-isomorphisms that are both natural and compatible with the Connes operators/de Rham differential.
Since $\Hoch(\Perfdg(\cA))$ is supported on $\Nonreg{\cA}$, the canonical map
$$
\R\G_{\Nonreg(\cA)} \Hoch(\Perfdg(\cA)) \xra{\sim} \Hoch(\Perfdg(\cA))
$$
is also a quasi-isomorphism, compatible with the Connes operator. 

We have thus proven the following result, which completes the proof of that portion of Theorem \ref{MainTheorem} that assumes $k$ is a $\Q$-algebra.

\begin{thm} \label{thm418}
For $k$ and $\cA$ as in Theorem \ref{MainTheorem}, assume also that $k$ is a $\Q$-algebra. There is a diagram of
dg $A$-modules of the form
$$
\Hoch(\Perfdg(\cA)) \xla{\sim} \bullet \xra{\sim} \bullet \xla{\sim} \bullet \xra{\sim} \R\Gamma_{\Nonreg(\cA)}(\Omega^\cdot_{A/k}, dw)
$$
such that
\begin{enumerate}
\item each map in the diagram is a quasi-isomorphism of dg $A$-modules,
  \item each map is natural for morphisms of curved algebra,
  \item each object in the diagram is equipped with a $k$-linear operator $B$ of homological degree $1$ and the maps commute with these operators,  and
  \item the operator $B$ on $\Hoch(\Perfdg(\cA))$ is the Connes operator and the operator $B$ on $\R\Gamma_{\Nonreg(\cA)}(\Omega^\cdot_{A/k}, dw)$ is induced from the de Rham
    differential.
  \end{enumerate}
\end{thm}

\section{When is $\mfdg(Q,f)$ smooth?} \label{sec:whensmooth}

For a dg category $C$, define $\cPerf(C^\op)$ to be the thick subcategory of $\cD(C^\op)$ (the derived category of right dg $C$-modules) generated by the representable modules.
We say a right dg $C$-module is {\em perfect} if it belongs to $\cPerf(C^\op)$.  As noted in \cite[\S 2.3]{Orlov}, a module is perfect if and only if it is isomorphic in the
derived category to a summand of a finite semi-free dg module and $\cPerf(C^\op)$ coincides with the subcategory of compact objects of $\cD(C^\op)$. 

Recall that for a $k$-linear dg category $C$, we define $C^e = C \otimes_k C^\op$ and
$\Deltaright = \Delta_C^\rght$ to be the right dg $C^e$-module given on objects by sending $(X, Y)$ to $\Hom_C(X, Y)$.

\begin{defn} \cite[3.23]{Orlov} A $k$-linear dg category $C$ is {\em homologically smooth (over $k$)} if the right dg $C^e$-module $\Deltaright$ is perfect.
\end{defn}

\begin{rem} There is a canonical isomorphism $(C^e)^\op \xra{\cong} C^e$ of dg category, which is the identity on objects
  and is given by $\a \otimes \b \mapsto (-1)^{|\a||\b|} \b \otimes \a$ on morphisms,
and under this isomorphism, $\Deltaright$ corresponds with $\Deltaleft$. 
It follows that $C$ is homologically smooth if and only if the left
dg $C^e$-module $\Deltaleft$ is perfect.
\end{rem}

Let us assume $k = F[t, t^{-1}]$ with $F$ an (ungraded) field and $t$ a degree two indeterminant, and recall that 
a dg $k$-module is the same thing as a $\Z/2$-graded complex of $F$
vector spaces, and more generally a $k$-linear dg category  is the same thing as a
$\Z/2$-graded dg category over $F$.  We say a $\Z/2$-graded dg category $C$ over $F$ is homologically smooth if $C$ is smooth over $k$.

Assume $Q$ is an essentially smooth $F$-algebra and $f \in Q$ a non-zero-divisor. Recall $\mfdg(Q,f)$ denotes the
$\Z/2$-graded dg category over $F$ of matrix factorizations of $f$; that is,
$\mfdg(Q,f) = \Perfdg(\cA)$ where we set $\cA = (Q[t, t^{-1}], ft)$.
Let $\Spec(Q) \xra{f} \A^1_F$ be the morphism of affine $F$-schemes induced by the map of $F$-algebras $F[x] \to Q$ sending $x$ to $f$.

\begin{thm} \label{smooththm}
  Assume $F$ is a perfect field, $Q$ is an essentially smooth $F$-algebra, and $f \in Q$ is a non-zero-divisor.
  The $\Z/2$-graded dg category $\mfdg(Q,f)$ is homologically smooth over $F$
  if and only if the origin in $\A^1_F$ is an isolated singular value of the morphism $\Spec(Q) \xra{f} \A^1_F$. In particular,
  $\mfdg(Q,f)$ is homologically smooth over $F$ whenever $\chr(F) = 0$. 
\end{thm}

\begin{proof} 
  Since $F$ is perfect, we have $\Nonreg(Q/f) = \Sing_F(Q/f)$ and thus $\Nonreg(Q/f) = V(f) \cap \Sing(f)$.
  The origin in $\A^1_F$ is an isolated singular value of $\Spec(Q) \xra{f} \A^1_F$
if and only if $\Nonreg(Q/f)$ is an open and closed subset of $\Sing(f)$.

Let us write $\LFdg(Q,f)$ for $\Moddg(\cA)$. Its objects are ``linear factorizations'' of $f$; that is, $\Z/2$-graded $Q$-modules (not necessarily finitely generated nor projective) equipped
with odd degree $Q$-linear endomorphisms that square to multiplication by $f$.

Set $Q^e = Q \otimes_F Q$ and $f^e = f \otimes 1 - 1 \otimes f$  and $W = \Nonreg(Q/f) \times_{\Spec F} \Nonreg(Q/f) \subseteq \Spec(Q^e)$.
Recall from Theorem \ref{thm47} that we have a Morita equivalence
  $$
  \phi: \mfdg(Q,f)^e \into \mfdg^W(Q^e, f^e)
  $$
  of $\Z/2$-graded dg categories, so that 
  extension and restriction of scalars along $\phi$ determines an equivalence of triangulated categories
  $\cD(\mfdg(Q,f)^e) \cong \cD(\mfdg^W(Q^e, f^e))$.
  This equivalence restricts to one on compact objects, and thus
  $$
  \cPerf(\mfdg(Q, f)^e) \cong \cPerf(\mfdg^W(Q^e, f^e)).
  $$
  Moreover, by \eqref{E1215}, the dg $\mfdg(Q,f)^e$-module  $\Deltaright$ corresponds to the dg $\mfdg^W(Q^e, f^e)$-module $h^Q$.
  (Recall we  view $Q$ as object of $\LFdg(Q^e, f^e)$ with trivial differential and $h^Q = \Hom_{\mfdg^W(Q^e, f^e)}(-, Q)$.
  Beware that $h^Q$ is not representable as a module on $\mfdg^W(Q^e, f^e)$.) 
It thus remains to prove 
  \begin{quote}
    $h^Q$ is a perfect dg $\mfdg^W(Q^e, f^e)$-module 
if and only if $\Nonreg(Q/f)$ is an open and closed subset of $\Sing(f)$. 
\end{quote}

Assume that $\Nonreg(Q/f)$ is an open and closed subset of $\Sing(f)$, so that
$\Sing(f)$ decomposes as a disjoint union $\Nonreg(Q/f) \amalg V$ of topological spaces. In other words,
there is a $g \in Q$ such that $V(g) \supseteq V$ and $V(g) \cap \Nonreg(Q/f) = \emptyset$.
Since all objects of $\mfdg(Q, f)$ are supported on $\Nonreg(Q/f)$, the canonical dg functor   $\mfdg(Q, f) \to \mfdg(Q[1/g], f)$
is fully-faithful, and it is essentially onto by Briggs'
Theorem \ref{body_thm_curved_regular}. Thus, it is a quasi-equivalence of $\Z/2$-graded dg categories and so, without loss of generality, 
we may thus assume that $\Sing(f) = \Nonreg(Q/f)$; i.e., that the origin is the only singular value of $\Spec(Q) \xra{f} \A^1_F$. In this case, 
we have
$$
\Nonreg(Q^e/f^e) = \Sing_F(Q^e/f^e) = V(f^e) \cap \Sing(f) \times_{\Spec(F)} \Sing(f) = \Sing_F(Q/f) \times_{\Spec F} \Sing_F(Q/f) = W
$$
so that
$$
\mfdg^W(Q^e, f^e) = \mfdg(Q^e, f^e). 
$$

As in the proof of Lemma \ref{lem416}, since $Q^e$ is regular and $Q$ is finitely generated, we can construct
an exact sequence $0 \to P_m \to \cdots \to P_0 \to Q \to 0$ with $P_i \in \mfdg(Q^e, f^e)$ for each $i$.
Setting $P = \Tot(0 \to P_m \to \cdots \to P_0 \to 0)$, the induced morphism $h^P \xra{\sim} h^Q$
of dg $\mfdg(Q^e, f^e)$-modules is a quasi-isomorphism. 
In particular, $h^Q$ is isomorphic in the derived category to a representable module, and hence is perfect.

Now assume $h^Q$ is a perfect dg $\mfdg^W(Q^e, f^e))$-module.
We claim that the homology of $\Z/2$-graded complex 
$\R \G_{V(f)}(\Omega^{\Z/2}_{Q/k}, df)$ is finitely generated as a $Q$-module.
Since the support of $(\Omega^{\Z/2}_{Q/k}, df)$ is equal to $\Sing(f: \Spec(Q) \to\A^1_F)$, we have that 
$\R \G_{V(f)}(\Omega^{\Z/2}_{Q/k}, df)$ has finitely generated homology if and only if $V(f) \cap \Sing(f)$ is an open and closed subset of $\Sing(f)$.
So, this claim will complete the proof of the theorem.

By Theorems \ref{thm711b} and \ref{thm47} (see also Corollary \ref{CorPerfectField}),  we  have an isomorphism
\begin{equation} \label{E1215c}
h^Q \Lotimes_{\mfdg^W(Q^e, f^e)}  t^Q \cong \R \G_{V(f)}(\Omega^{\Z/2}_{Q/k}, df)
\end{equation}
in the derived category of $\Z/2$-graded chain complexes of $Q$-modules.
Since we assume $h^Q$ is perfect, it is isomorphic in the derived category to a finite semi-free
dg $\mfdg^W(Q^e, f^e)$-module $G$. The isomorphism \eqref{E1215c} thus gives that
$\R \G_{V(f)}(\Omega^{\Z/2}_{Q/k}, df)$ is a summand in $\cD(k)$ of
$$
G  \Lotimes_{\mfdg^W(Q^e, f^e)}  t^Q
\cong t^Q(X) = 
G  \otimes_{\mfdg^W(Q^e, f^e)}  t^Q
$$
It therefore suffices to show the homology $G  \otimes_{\mfdg^W(Q^e, f^e)}  t^Q$ is finitely generated as a $Q$-module. 

Say $0 = G_{-1} \subseteq G_0 \subseteq \cdots \subseteq G_m = G$ is a bounded filtration with $G_i/G_{i+1}$ a finite free dg $\mfdg^W(Q^e, f^e)$-module
for each $i$. For $X \in \mfdg^W(Q^e, f^e)$ by \eqref{E1029a} we have
$$
h^X  \otimes_{\mfdg^W(Q^e, f^e)}  t^Q \cong Q \otimes_{Q^e} X,
$$
which clearly has finitely generated homology. It follows that the homology of 
$G_i/G_{i-1}  \otimes_{\mfdg^W(Q^e, f^e)}  t^Q$
is finitely generated for each $i$ and, by induction, that the homology of 
$G  \otimes_{\mfdg^W(Q^e, f^e)}  t^Q$ is finitely generated. 
\end{proof}


\appendix

\section{Thick subcategories of categories of curved dg modules} \label{appendix}



A famous result of Hopkins \cite{Hopkins} and Neeman \cite{NeemanChromatic} classifies thick subcategories of the perfect derived category of a noetherian ring in terms of their prime support. In this appendix we prove that the analogue for curved dg modules holds by reducing to the case of ordinary dg modules.

\subsection{Recollections on dg modules}\label{sec_A_dgmod} If $A$ is a graded ring, we write $\Moddg(A)$ for the dg category of all dg $A$-modules. The complex of maps between two objects $X$ and $Y$ is denoted $\Hom_A(X,Y)$. The associated homotopy category $[\Moddg(A)]$ has the same objects and hom sets
\[
\Hom_{[\Moddg(A)]}(X,Y) \coloneqq H^0 (\Hom_A(X,Y)).
\]
The {\em derived category} $\cD(A)$ of all dg $A$-modules is defined as the Verdier localization of $[\Moddg(A)]$ that inverts all quasi-isomorphisms. 

We write $\projdg(A)$ for the full dg subcategory of $\Moddg(A)$ consisting of dg modules that are finitely generated and projective as graded $A$-modules. The corresponding homotopy category of finitely generated projective dg modules is written $[\projdg(A)]$.

A dg $A$-module $X$ having the property  that $\Hom_{A}(X, -)$ preserves all quasi-isomorphisms is called {\em homotopically projective}. A dg $A$-module is called {\em perfect} if it is homotopically projective and in $\projdg(A)$. We write $\Perfdg(A)$ for the full dg subcategory of $\Moddg(A)$ consisting of perfect dg $A$-modules.

When $X$ is homotopically projective, the canonical map $H^0(\Hom_{A}(X, Y)) \to \Hom_{\cD(A)}(X, Y)$ is an isomorphism for all dg modules $Y$ by \cite[Proposition 2.3.3]{Ver}. In particular, we may regard $[\Perfdg(A)]$ as a full subcategory of $\cD(A)$.

A full subcategory of a triangulated category $\cT$ is  {\em thick} if it is triangulated and closed under summands. The smallest thick subcategory of $\cT$ that contains an object $X$ is denoted $\Thick_\cT(X)$. We will use the short-hand notation $Y \models_\cT X$ to mean  that $X$ is in $\Thick_\cT(Y)$.

\begin{rem}\label{rem_idem_complete} 
As defined here, the homotopy category $[\Perfdg(A)]$ may not be idempotent complete. The embedding $[\Perfdg(A)] \into  \Thick_{\cD(A)}(A)$ realizes $\Thick_{\cD(A)}(A)$ as the idempotent completion of $[\Perfdg(A)]$, and in particular the two dg categories are Morita equivalent.

An analogous phenomenon was discovered in homotopy theory by Wall, who used obstructions in $K$-theory to  show that there are spaces that are, up to homotopy, retracts of finite CW-complexes, but that are not themselves homotopy equivalent to finite CW-complexes \cite{Wall}.

Taking inspiration from this, it follows from Thomason's theorem \cite{Thom} that $[\Perfdg(A)]$ is idempotent complete if and only if the map $K_0([\Perfdg(A)])\to   K_0(\Thick_{\cD(A)}(A))\eqqcolon K_0(A)$ is an isomorphism; this happens exactly when $K_0(A)$ is generated by the classes of projective graded $A$-modules. For example, if $A=A^{\leqslant 0}$ then there is a canonical isomorphism $K_0(A)\cong K_0(A_0)$ by \cite[Lemma 2.4]{LaTa}, and so for these graded rings $[\Perfdg(A)]$ is idempotent complete.
\end{rem}

\subsection{Support for dg modules}

When $A$ is a graded commutative ring, the set of homogeneous prime ideals of $A$ is denoted $\Spec^*(A)$. 
If $\fp\in \Spec^*(A)$ then the graded localization $A_{(\fp)}$ is the result of inverting all homogeneous elements of $A \smallsetminus \fp$. 
If $X$ is a graded $A$-module then $X_{(\fp)}$ is the graded $A_{(\fp)}$-module $X\otimes_A A_{(\fp)}$, and we define the graded support of $X$ to be
\[
    \supp^*_A(X) \coloneqq \{\fp \in \Spec^*(A) ~\mid~ X_{(\fp)} \ne 0\text{ in }\cD(X)\}.
\]

Extending this slightly, if $X$ is a dg $A$-module then $\supp^*(X)= \supp^*_A(H(X))$, or equivalently, the set of homogeneous primes $\fp$ such that $X_{(\fp)}$ is not isomorphic to zero in $\cD(A)$.


The next result, due to Carlson and Iyengar, is an analogue of Hopkins and Neeman's theorem for dg modules.

\begin{thm}[\cite{CI}]\label{HopkinsTheorem} Let $A$ be  a graded commutative noetherian ring. If $X$ and  $Y$ are in $\Thick_{\cD(A)}(A)$ and $\supp^*_A(X) \subseteq \supp^*_A(Y)$, then  $X$ is in $ \Thick_{\cD(A)}(Y)$.
\end{thm}

In \cite{CI} $A$ is required to be concentrated in even degrees or to satisfy $2A$=0; this doesn't seem to be necessary in their argument except to guarantee that $A$ is commutative in the ordinary sense.

For the sake of novelty we give a different proof of Theorem \ref{HopkinsTheorem}, explaining how it can be quickly deduced from an earlier result of Dell’Ambrogio and Stevenson on the graded derived category \cite{DS}.

\begin{proof}
We work with the abelian category $\operatorname{grMod}(A)$ of graded $A$-modules, and we denote its derived category by  $\cD(\operatorname{grMod}(A))$. As an excuse to fix our grading notation: a complex $C$ of graded modules has a differential $d_C^{i,n} \colon C^{i,n}\to C^{i,n+1}$; it has a  suspension $\Sigma^nC$ with $(\Sigma^nC)^{i,m}=C^{i,m+n}$ and $d_{\Sigma^n C}^{i,m}=(-1)^nd_C^{i,m+n}$; and it has a twist $C(j)$ with $C(j)^{i,n}=C^{i+j,n}$ and $d_{C(j)}^{i,n}=d_C^{i+j,n}$.

We need to set up two functors:
\[
  \xymatrix{
        \cD(A) \ar@<1mm>[r]^(.35){G} &\cD(\operatorname{grMod}(A)).\ar@<1mm>[l]^(.65){F}
    }
\]

The functor $G$ (for grade) is defined by 
\[
G(X)^{i,n}=X^{i+n} \quad\text{with}\quad d_{G(X)}^{i,n}=(-1)^i d_X^{i+n} \colon X^{i+n} \to X^{i+n+1}.
\]
The sign is needed to make $G(X)$ a complex of graded $A$-modules. 

The functor $F$ (for fold) is defined by 
\[
F(C)^m=\bigoplus_{i+n=m}C^{i,n} \quad\text{with}\quad d^n_{F(C)}={\textstyle\sum_{i+n=m}}[(-1)^i d^{i,n}_C\colon C^{i,n}\to C^{i,n+1}].
\]
Again, the sign is needed to make $F(C)$ a dg module. 
We note that $F(C(i))=\Sigma^iF(C)$.  

Both $F$ and $G$ preserve quasi-isomorphisms and so induce well-defined functors on the corresponding derived categories. 

The functor $G$ preserves support, in the sense that $\supp_A(G(X))\subseteq \supp_A(G(Y))$. By Dell’Ambrogio and Stevenson's theorem \cite[Theorem 5.8]{DS} $G(X)$ is in the localising subcategory generated by $G(Y)(i)$ for $i\in \mathbb{Z}$. (We note that the results of \cite{DS} are stated for \emph{small support}, but since $X$ and $Y$ have finitely generated homology, the support sets of $G(X)$ and $G(Y)$ are the same as their small support sets.) 

The functor $F$ preserves coproducts, and it follows that $FG(X)= X^{\oplus{\mathbb{Z}}}$ is in the localising subcategory generated by $F(G(Y)(i))=\Sigma^iFG(Y)=\Sigma^i(Y^{\oplus{\mathbb{Z}}})$, for $i\in \mathbb{Z}$. This implies that $X$ is in the localising subcategory generated by $Y$. Since $X$ and $Y$ are compact objects of $\cD(A)$, Neeman's lemma \cite[Lemma 2.2]{NeemanKtheory} implies that  $X$ is in $ \Thick_{\cD(A)}(Y)$.
\end{proof}



\subsection{Curved dg modules}  A curved ring $\cA=(A,w)$ is a graded ring $A$ equipped with a central element $w\in A^2$. A curved dg $\cA$-module is a graded $A$-module $X$ equipped with a degree one differential $d\colon X\to X$ that is $A$-linear in the sense that $d(ax)=(-1)^{|a|}ad(x)$ when $a\in A$ and $x\in X$, and such that $d^2$ is multiplication by $w$.

If $X$ and $Y$ are curved dg $\cA$-modules then the set of $A$-linear maps $\Hom_A(X,Y)$ is a complex, given its usual differential. This makes the collection of all curved dg $\cA$-modules into a pretriangulated dg category, denoted $\Moddg(\cA)$, with homotopy category $[\Moddg(\cA)]$.

As in Section \ref{sec_A_dgmod}, we write $\projdg(\cA)$ for the full dg subcategory of $\Moddg(\cA)$ consisting of curve  dg modules that are finitely generated and projective as graded $A$-modules. The corresponding homotopy category is written $[\projdg(\cA)]$.

\begin{rem} The homotopy category of (not necessarily finitely generated) projective dg modules has been studied in a number of contexts \cite{NeemanFlat,IyengarKrause}, including under the names contraderived category, derived category of the second kind, and unseparated derived category, and including in the curved setting; see the references in \cite{PositselskiStovicek}. In general it seems best to view $[\projdg(\cA)]$ within this context, however we will see that in the regular context it is reasonable to view $[\projdg(\cA)]$ as an analogue of the perfect derived category.
\end{rem}

Suppose that $A$ is graded commutative. Then $\Hom_A(X,Y)$ is further a dg $A$-module whenever $X$ and $Y$ are curved dg $\cA$-modules. In this case the support of a curved dg $\cA$-module $X$ is defined to be the support of its endomorphism dg module:
\[
\supp^*_A(X) \coloneqq \supp_A^*(\Hom_A(X,X)).
\]
In a similar vein, $X$ is a curved dg $A$-module (with zero curvature) and $Y$ is a dg $\cA$-module, then the tensor product $X\otimes_AY$ is a curved dg $\cA$-module, given its usual differential.

\begin{thm} \label{BriggsTheorem}
  Let $\cA = (A, w)$ be a curved ring such that $A$ is noetherian and graded commutative. 
Assume $X, Y \in \projdg(\cA)$ are such that each of the dg $A$-modules $\Hom_A(X,X)$, $\Hom_A(Y,Y)$ and $\Hom_A(Y,X)$ are perfect. 
If $\supp^*_A(X) \subseteq \supp^*_A(Y)$ then  $X $ is in $ \Thick_{[\projdg(\cA)]}(Y)$.
\end{thm}

\begin{proof}
Taking $X$ and $Y$ as stated, we have
\[
\supp^*_A(\Hom_A(X,X)) \subseteq \supp^*_A(\Hom_A(Y,Y))
\]
by definition. Since $\Hom_A(X,X)$ and $\Hom_A(Y,Y)$ are perfect dg $A$-modules, they in particular lie in $\Thick_{\cD(A)}(A)$, so from Theorem \ref{HopkinsTheorem} we obtain
     \begin{equation}\label{app_eqn_build}
         \Hom_A(Y,Y)\models_{\cD(A)}\Hom_A(X,X).
     \end{equation}
We wish to deduce that $Y\models_{[\projdg(\cA)]}X$.

By assumption we have $A \models_{\cD(A)} \Hom_A(X,Y)$. Since both $A$ and $\Hom_A(X,Y)$ lie in $[\Perfdg(A)]$, an application of \cite[Proposition 2.3.1]{Ver} yields $A \models_{[\Perfdg(A)]} \Hom_A(X,Y)$ (despite the fact that $[\Perfdg(A)]$ as defined in \ref{sec_A_dgmod} may not be idempotent complete in $\cD(A)$).

Applying the exact functor $-\otimes_A Y\colon [\Perfdg(A)]\to [\projdg(\cA)]$ yields
\[
Y \models_{[\projdg(\cA)]} \Hom_A(Y, X) \otimes_A Y \cong \Hom_A(Y,Y)\otimes_A X.
\]

We similarly use \cite[Proposition 2.3.1]{Ver} on (\ref{app_eqn_build}), and then apply $-\otimes_A X\colon [\Perfdg(A)]\to [\projdg(\cA)]$ to obtain
\[
 \Hom_A(Y,Y)\otimes_A X \models_{[\projdg(\cA)]}  \Hom_A(X,X)\otimes_A X.
\]
Finally, $X$ is a summand of $\Hom_A(X,X)\otimes_A X$ in $[\Perfdg(\cA)]$; indeed, the composition of the map $X \to \Hom_A(X,X) \otimes_A X$, $x\mapsto \id \otimes x$ with the map $\Hom_A(X,X) \otimes_A X\to X$, $\theta \otimes x \mapsto \theta(x)$ is the identity. Combining this with the already obtained inclusions of thick subcategories yields the desired result $Y\models_{[\projdg(\cA)]}X$.
\end{proof}

\subsection{Curved dg modules over regular rings} 
A graded commutative ring $A$ is \emph{regular} if it is noetherian and for every homogeneous prime $\fp\in \Spec^*(A)$, the graded local ring $A_{(\fp)}$ has finite global dimension (i.e.\ all of its graded modules have finite projective dimension). The non-regular locus of a graded ring $A$ is the set $\Nonreg^*(A)$ of homogeneous primes $\fp$ such that $A_{(\fp)}$ is not a regular ring. We note that if $A$ is local and regular then $A$ must either be evenly graded or satisfy $2A=0$. The next lemma is well known; it explains why curved dg modules over regular rings are well behaved.

\begin{lem}\label{lem_perf_curved}
    Let $A$ be a regular graded commutative ring. If  $X$ is a dg $A$-module that is finitely generated and projective as a graded $A$-module, then $X$ is perfect.

    Let $\cA = (A, w)$ be a curved ring such that $A$ is graded commutative and regular. If $X$ and $Y$ are curved dg $\cA$-modules that are finitely generated and projective as a graded $A$-modules, then $\Hom_A(X,Y)$ is a perfect dg $A$-module.
\end{lem}

\begin{proof}
Taking a dg module $X$ that is finitely generated and projective as a graded module, our goal is to show that $X$ is homotopically projective. Since $X$ is finitely generated and projective, $\Hom_A(X,Y)_{(\fp)}\cong \Hom_{A_{(\fp)}}(X_{(\fp)},Y_{(\fp)})$ for any dg module $Y$ and any homogeneous prime $\fp$. Therefore we may localise at $\fp$ and assume that $A$ has finite global dimension. In this situation, projective dg modules are homotopically projective by \cite[3.1]{Dold} (this result is stated for ungraded rings, but the argument for dg modules is similar).

The second second statement follows from the first, since $\Hom_A(X,Y)$ is finitely generated and projective as a graded $A$-module.
\end{proof}

\begin{defn}\label{perf_def_appendix_curved} 
Let $\cA = (A, w)$ be a curved ring such that $A$ is graded commutative and regular. Based on Lemma~\ref{lem_perf_curved}, curved dg $\cA$-modules that are finitely generated and projective over $A$ will be called \emph{perfect}, and we will use the notation $\projdg(\cA)=\Perfdg(\cA)$. We again emphasise that it only seems reasonable to use this terminology when $A$ is regular. Compare with Section \ref{introduction} above.
\end{defn}

For a homogeneous prime $\fp$ of $A$, we write $\cA_{(\fp)}$ for the curved ring $(A_{(\fp)}, \frac{w}{1})$ and, for a curved dg $\cA$-module $X$, we let $X_{(\fp)}$ denote the curved dg $\cA_{(\fp)}$-module $X\otimes_A A_{(\fp)}$.

\begin{lem} \label{applem2}
  Let $\cA = (A, w)$ be a curved ring such that $A$ is graded commutative and regular, and let $X$ be a perfect curved dg $\cA$-module. Then
  \[
  \supp^*_A(X) \cap \cV(w)= \Big\{ \fp  \mid w\in\fp \text{ and $X_{(\fp)}$ is not isomorphic to zero in $[\Perfdg(\cA_{(\fp)})]$}\Big\},
  \]
  where $\cV(w)= \{\fp\mid w\in\fp\}$. If $A$ is evenly graded then $\supp^*_A(X) \cap \cV(w)= \supp^*_A(X)$.
\end{lem}

\begin{proof} 
Since $X$ is projective and finitely generated, $\Hom_{A_{(\fp)}}(X_{(\fp)},X_{(\fp)})\cong\Hom_A(X,X)_{(\fp)}$ for any $\fp\in \Spec^*(A)$. Therefore $X_{(\fp)}$ is isomorphic to zero in $[\Perfdg(\cA_{(\fp)})]$ if and only if $\Hom_A(X,X)_{(\fp)}$ is contractible. By Lemma \ref{lem_perf_curved} $\Hom_A(X,X)_{(\fp)}$ is homotopically projective, so it is acyclic if and only if it is contractible. This establishes the first claim.

If $w\notin \fp$ then $w\colon X_{(\fp)}\to X_{(\fp)}$ is invertible, and it follows that the differential $d\colon X_{(\fp)}\to X_{(\fp)}$ is invertible as well, with $A_{(\fp)}$-linear inverse $(d)^{-1}$. Define $h\colon X_{(\fp)}\to X_{(\fp)}$ by the rule $h(x)=(d)^{-1}(x)$ is $|x|$ is even and $h(x)=0$ is $|x|$ is odd, and note that $hd+dh=\id_{X_{(\fp)}}$.  If $A$ is evenly graded then so is $A_{(\fp)}$, and this implies that $h$ is $A_{(\fp)}$-linear. It follows that $\Hom_A(X,X)_{(\fp)}\cong \Hom_{A_{(\fp)}}(X_{(\fp)},X_{(\fp)})$ is contractible. This argument shows that $\supp^*_A(X)\subseteq \cV(w)$.
\end{proof}

\begin{rem}
The hypothesis that $A$ is evenly graded is necessary to ensure that $\supp^*_A(X) \subseteq \cV(w)$. 
For example, let $A=k[x^{\pm1}]$, where $k$ is a field of characteristic $2$ and $x$ is a variable  of degree $1$, and set $w=x^2$. The perfect curved dg module $X=A$, with differential $d_X=x$, is supported at the unique homogeneous prime $\fp=(0)$ of $A$, despite the fact that $w\notin \fp$.
\end{rem}

\begin{lem}\label{lem_reg_zero}
Let $\cA = (A, w)$ be a curved local ring. If $w$ is a non-zero-divisor (and not a unit) 
and the quotient $A/(w)$ is regular, then $[\projdg(\cA)]\simeq 0$.
\end{lem}

\begin{proof}
   Let $X$ be a curved dg $\cA$-module that is finitely generated and projective, and set $Y = \End_A(X)$. 
   
   Since $w$ is a non-zero-divisor, it follows that the complex $X\otimes_A(A/(w))$ is exact. Since $A/(w)$ is regular, by lemma \ref{lem_perf_curved} the dg module $X\otimes_AA/(w)$ is $h$-projective, and therefore contractible. In other words, $\End_{A/(w)}(X\otimes_AA/(w))= Y\otimes_AA/(w)$ is exact.

   From the exact sequence in homology associated to the triangle $\Sigma^{-2} Y \xrightarrow{w} Y \to Y\otimes_AA/(w)$, we find that $\Sigma^{-2} {\rm H}(Y) \xrightarrow{w} {\rm H}(Y) $ is an isomorphism. Then by Nakayama's lemma ${\rm H}(Y)=0$. This means that $X$ is the zero object in $[\projdg(\cA)]$, as required.
\end{proof}

We will say that $\cA=(A,w)$ is \emph{regular} at $\fp\in \Spec^*(A)$ if $\frac{w}{1}$ is a non-zero-divisor in $A_{(\fp)}$ and the quotient $A_{(\fp)}/(\frac{w}{1})$ is a regular local ring. The set of homogeneous primes at which $\cA$ is not regular is written $\Nonreg^*(\cA)$. If $w$ is assumed to be a non-zero-divisor then we may make the identification $\Nonreg^*(\cA)= \Nonreg^*(A/(w))$. 

 We also remind the reader that a subset $Z\subseteq \Spec^*(A)$ is called \emph{specialization closed} if $\fp\in Z$ and $\fp\subseteq \fq$ implies $\fq\in Z$, for any $\fp,\fq\in \Spec^*(A)$.

\begin{thm}\label{thm_curved_regular}
Let $\cA = (A, w)$ be a curved ring such that $A$ is graded commutative and regular. If $X$ and $Y$ are in $\Perfdg(\cA)$ and $\supp_A^*(X) \subseteq \supp_A^*(Y)$, then $X$ is in $ \Thick_{[\Perfdg(\cA)]}(Y)$. When $A$ is evenly graded this yeilds a bijection
  \[
\xymatrix{
\Big\{\text{thick subcategories of }[\Perfdg(\cA)]\Big\} \ar@<1mm>[r]^(.45){\sigma} & \Big\{\text{specialization closed subsets of }\Nonreg^*(\cA)\Big\}, \ar@<1mm>[l]^(.55){\theta}
}   
 \]
where $\theta(Z)=\{X\in [\Perfdg(\cA)]  \mid \supp_A^*(X)\subseteq Z\}$ and $\sigma(T)=\bigcup_{X \in T} \supp^*_A(X)$. 
\end{thm}

\begin{proof}
The first statement is obtained by combining Theorem \ref{BriggsTheorem} with Lemma \ref{lem_perf_curved}.


The assignment $X \mapsto X_{(\fp)}$ determines an exact functor $[\Perfdg(\cA)]\to[\Perfdg(\cA_{(\fp)})]$, and therefore its kernel $\theta(\fp)$ is a thick subcategory of $[\Perfdg(\cA)]$. It follows from Lemma \ref{applem2} that $\theta(Z)=\bigcap_{\fp\in Z}\theta(\fp)$ is also a thick subcategory of $[\Perfdg(\cA)]$.

Suppose that  $A$ is evenly graded. The support of any perfect curved dg module is then a Zariski closed subset of $\Nonreg^*(\cA)$ by Lemmas \ref{applem2} and  \ref{lem_reg_zero}. Since unions of Zariski closed subsets are specialization closed, it follows that $\sigma(T)$ is a specialization closed subset of $\Nonreg^*(\cA)$ for any thick subcategory $T$ of $[\Perfdg(\cA)]$.

The already established first statement implies that $\sigma \theta =\id$. Therefore, it remains to show that $\sigma$ is surjective to obtain the claimed bijection. For this, we give a construction borrowing inspiration from \cite[2.3]{BGS}. 

Since $A$ is regular it is a finite product $A_1\times\cdots\times A_m$ of regular graded domains, and the curvature element may be written $w=(w_1,\ldots ,w_m)$. As each $A_i$ is a projective graded $A$-module there are restriction functors $\rho_i\colon [\Perfdg(A_i,w_i)]\to[\Perfdg(\cA)]$, and through the inclusions $\Spec^*(A_i)\subseteq \Spec^*(A)$ these  satisfy $\supp_{A_i}^*(X)=\supp_{A}^*(\rho_i(X))$. Hence we may assume that $A$ is a domain.

Every specialization closed subset is a union of irreducible Zariski closed subsets, so it suffices to show that for every $\fp\in  \Nonreg^*(\cA)$ there is a perfect curved dg module $X$ with support $\cV(\fp)=\{\fq \mid \fp\subseteq \fq\}$. 

Let $K$ be the Koszul complex over $A$ on a generating set $x_1,\ldots,x_n$ for $\fp$, thought of a dg $A$-algebra generated by exterior variables $e_1,\ldots, e_n$ with its Koszul differential $d(e_i)=x_i$. 

Since $\fp\in  \Nonreg^*(\cA)$ either $\frac{w}{1}=0$ in $A_{(\fp)}$ or $A_{(\fp)}/(\frac{w}{1})$ is not regular, and in either case $\frac{w}{1}\in \fp^2A_{(\fp)}$. Since $A$ is a domain, it follows that $w\in \fp^2$, and we may write $w=x_1y_1+\cdots + x_ny_n$ with $y_i\in \fp$.

Multiplication by the element $u=e_1y_1+\cdots + e_ny_n$ defines a map $u\colon K^*\to K^{*+1}$, and a computation shows that $(d+u)^2=w$. Hence $X=(K,d+u)$ is a perfect curved dg $\cA$-module.

It remains only to verify that  $\supp^*_A(X)=\cV(\fp)$. If $\fp\subseteq \fq$ then $X\otimes_A A_{(\fq)}/(\fq)$ is an exterior algebra over  $A_{(\fq)}/(\fq)$ with trivial differential, and in particular $X_{(\fq)}$ cannot be isomorphic to zero, so $\fq\in\supp^*_A(X)$ by Lemma \ref{applem2}. Conversely if $\fp\not\subseteq \fq$ then there is an $i$ such that $x_i\notin \fq$. The map $h=\frac{e_i}{x_i}\colon X^*_{(\fq)}\to X^{*-1}_{(\fq)}$ then satisfies $(d+u)h+h(d+u)=dh+hd=\id_X$, showing that $X_{(\fq)}$ is isomorphic to zero in $[\Perfdg(\cA_{(\fq)})]$, and so $\fq\notin\supp^*_A(X)$ by Lemma \ref{applem2}.
\end{proof}

An object $X \in [\Perfdg(\cA)]$ is said to be a classical generator if $\Thick_{[\Perfdg(\cA)]}(X)=[\Perfdg(\cA)]$. The next result follows directly from Theorem \ref{thm_curved_regular}. It is a curved analogue of the main result in \cite{IR}. We note that if $A$ is excellent
then $\Nonreg^*(\cA)$ is Zariski closed.


\begin{cor} \label{corA12} If $\cA = (A, w)$ is a curved ring such that $A$ is evenly graded, commutative and regular, then the following are equivalent:
\begin{enumerate}
    \item $\Nonreg^*(\cA)$ is a  Zariski closed subset of $\Spec^*(A)$;
    \item\label{item_gen}there is an object $X$ of $[\Perfdg(\cA)]$ whose support is $\Nonreg^*(\cA)$;
    \item $[\Perfdg(\cA)]$ admits a classical generator;
    \item\label{item_morita} $[\Perfdg(\cA)]$ is equivalent, up to summands, to $[\Perfdg(R)]$, for some dg algebra $R$.
\end{enumerate}
Any object $X$ as in (\ref{item_gen}) generates $[\Perfdg(\cA)]$, and then we may take $R=\End_A(X)$ in (\ref{item_morita}).
\end{cor}

We end the appendix by highlighting several special cases.

\begin{ex}[Hypersurface rings]
Let $Q$ be a commutative noetherian ring and let $f \in Q$. We cosider the curved ring $\cA = (A, w) = (Q[t^{\pm1}], ft)$ where   $t$ is a degree two indeterminant. The dg category $\Perfdg(\cA)$  may be identified with  $\mfdg(Q,f)$, the dg category of matrix factorizations of $f$.
  
When $Q$ is regular and $f$ is a non-zero-divisor, the corresponding homotopy category $[\mfdg(Q,f)]$ is equivalent to the singularity category to $\cD^\sing(R) = \frac{\cD^{\rm b}(R)}{\Perfdg(R)}$ of the hypersurface ring $R = Q/f$; this was proven by Buchweitz \cite[Theorem 4.4.1]{MCMbook}, and later rediscovered by Orlov \cite[Theorem 3.9]{Orlov}.
  
We may also identify $\Nonreg^*(\cA)$ with $\Nonreg(R)$, and thus the Theorem \ref{thm_curved_regular} yields a bijection
\[
\Big\{\text{thick subcategories of }\cD^\sing(R)\Big\}      \xleftrightarrow{\ \cong\ } \Big\{\text{specialization closed subsets of }\Nonreg(R)\Big\}
 \]
 This has been proven before by Takahasi \cite{Tak}, and later, using different methods, by Stevenson \cite[7.9]{Stevenson}. Hirano has also established a more general classification result for matrix factorizations over non-regular rings  \cite{Hirano}. This result suggests that the ``homotopically projective'' assumption in Theorem \ref{BriggsTheorem} may be unnecessary (but in this appendix we limit ourselves to results that can be easily deduced from the existing literature).
 
When $R$ is local with an isolated singularity, that is, when $\Nonreg(R)$ consists of only the maximal ideal $\fm$, it follows that $\cD^\sing(R)$ is the thick closure of the class of the residue field $R/\fm$; this was first proven by Dyckerhoff \cite{Dyckerhoff} (in the case that $Q$ is equicharacteristic).
\end{ex}

\begin{ex}[Complete intersection rings]
Let $f_1, \dots, f_c$ be a regular sequence in $Q$, and set $R = Q/(f_1, \dots, f_c)$. This time we consider the curved ring $\cA =(A,w)$, where $A=Q[t_1, \ldots, t_c]$ for some degree two indeterminants $t_1, \ldots, t_c$, and where $w= f_1t_1+\cdots + f_ct_c$.

Since $A$ is evenly graded, every perfect curved dg module $X$ over $\cA$ splits as a direct sum of two projective graded $A$-modules $X=X_{\rm even}\oplus X_{\rm odd}$, and this allows us to identify $[\Perfdg(\cA)]$ with the homotopy category graded $[\mathrm{gr\text{-}mf}(A,w)]$  of graded matrix factorizations of $w$, in the sense of \cite{BurkeStevenson}. Therefore, in our language, \cite[7.5]{BurkeStevenson} says that there is an equivalence of triangulated categories $[\Perfdg(\cA)] \cong D^b(R)$. Theorem \ref{thm_curved_regular} now yields a bijection
\[
\Big\{\text{thick subcategories of }\cD^{\rm b}(R)\Big\}      \xleftrightarrow{\ \cong\ } \Big\{\text{specialization closed subsets of }\Nonreg^*\!\Big(\frac{Q[t_1, \dots, t_c]}{f_1t_1+\cdots +f_ct_c}\Big)\Big\}.
\]
Observe that $\Spec^*(A)$ is the union of the closed subset $\cV^*(t_1, \ldots, t_c)$, which may be identfied with $\Spec(Q)$, 
and its open complement, which may be identified with $\bP^{c-1}_Q$. It follows that the nonregular locus of $A/(w)$ contains $\Spec(R)$ as a closed subset and its open complement is the nonregular locus of the projective hypersurface  $Y := \Proj \left(A/(w)\right)$ in $\bP^{c-1}_Q$.

Burke and Walker prove that $\cD^{\rm sing}(R)$ is equivalent to $[\mf(\bP^{c-1}_Q, \cO(1),w) ]$, the homotopy category of matrix factorizations of $w$, thought of as a section of  $\cO(1)$ \cite{BurkeWalker2}. It follows from \cite[Theorem 1]{BurkeWalker1} that $[\mf(\bP^{c-1}_Q, \cO(1),w) ]$ is the Verdier quotient of  $[\mathrm{gr\text{-}mf}(A,w)]$ at the thick subcategory of those objects supported $\cV^*(t_1, \dots, t_c)\subseteq \Spec^*(A)$. The commutative diagram at the end of \cite{BurkeStevenson} then identifies this thick subcategory with $\Perfdg(R)$, under the equivalence $[\Perfdg(\cA)] \cong D^b(R)$. Therefore, after taking the Verdier quotient we obtain a bijection 
\[
\Big\{\text{thick subcategories of }\cD^\sing(R)\Big\}      \xleftrightarrow{\ \cong\ } \Big\{\text{specialization closed subsets of }\Nonreg(Y)\Big\}.
\]
This classification was obtained by Stevenson in \cite[10.5]{Stevenson}.
\end{ex}

\bibliographystyle{amsplain}
\bibliography{mergedbib}

\end{document}